%% file: main.tex
\documentclass{article}
\usepackage{arXiv}

\usepackage{todonotes}

\usepackage{fancyhdr}
\title{
    \textbf{
        Analysis of the Ill-Conditioning of the Discrete Inverse Laplace Transform in Monte Carlo Simulations of Quantum Many-Body Systems
    }
}

\author[1]{Phil-Alexander Hofmann\thanks{\href{mailto:p.hofmann@hzdr.de}{p.hofmann@hzdr.de}}}
\author[2,1]{Thomas Chuna}
\author[2,1]{Tobias Dornheim}
\author[3]{Michael Hecht}

\affil[1]{Center for Advanced Systems Understanding, 
Helmholtz-Zentrum Dresden-Rossendorf, 
02826 Görlitz, 
Germany}

\affil[2]{Institute of Radiation Physics, 
Helmholtz-Zentrum Dresden-Rossendorf, 
01328 Dresden, 
Germany}

\affil[3]{Mathematical Institute, 
University of Wroc{\l}aw, 
50-384 Wroc{\l}aw, 
Poland}

\begin{document}


\maketitle

\begin{abstract}
    %
    %
    Analytic continuation of imaginary-time quantum Monte Carlo data to real-frequency spectra requires the inversion of a severely ill-posed two-sided Laplace transform and arises naturally in quantum many-body calculations of dynamic properties. 
    %
    %
    In this work, we distinguish the intrinsic ill-posedness of the continuous inverse two-sided Laplace problem from the conditioning of its finite-dimensional discretization.
    For equidistant sampling and reconstruction grids, we express the discrete problem via a diagonally scaled monomial Vandermonde matrix with exponentially distributed nodes. 
    Imposing the physical detailed-balance symmetry transforms the discretization into a diagonally scaled Chebyshev--Vandermonde system.
    Exploiting these structures, we derive explicit lower and upper bounds on the condition numbers in terms of the physical and discretization parameters.
    %
    %
    For the unconstrained discretization, our bounds reveal super-exponential growth of the condition number with the reconstruction dimension, of the form \(\exp(\mathcal{O}(n\log n))\), which cannot be removed by increasing the number of imaginary-time samples alone.
    Detailed-balance substantially improves the conditioning, especially in the practically relevant pre-asymptotic regime, although the asymptotic super-exponential dependence remains.
    In both cases, our bounds identify a low-dimensional regime in which the super-exponential contribution is suppressed and stable reconstruction remains feasible.
    %
    %
    These results provide a mathematical explanation for the effectiveness of low-dimensional spectral representations and detailed-balance in analytic continuation. 
    While they do not remove the fundamental ill-posedness, they substantially mitigate the ill-conditioning of its finite-dimensional discretization, delay the onset of its catastrophic super-exponential growth and thereby enlarge the range of numerically accessible reconstructions.
\end{abstract}



\section{Introduction}
\label{sec:introduction}

Motivated by the inverse two-sided Laplace transform problem arising in \emph{quantum Monte Carlo} (QMC) simulations of quantum many-body systems~\cite{Boninsegni2006a, Filinov2012, Nolting2009, Krilov1999, Rabani2002}, we study the numerical problem of recovering a discrete function \(f(t)\) from its discrete two-sided Laplace transform \(F(s)\),
\begin{equation}\label{eq:two_side_laplace_transform}
    F(s) = \int_{-\infty}^\infty \exp(-st) f(t) \, dt \, ,
\end{equation}
where \(F(s)\) is given on the real interval \(s \in [0,\beta]\). Though obtaining a reliable estimate of \(f(t)\) is difficult, solving this problem reliably would enable the estimation of dynamic material properties in a vast number of physical systems, ranging from solid state physics applications~\cite{Mishchenko2000,ShaoSandvik_PhysRep_2023} to ultracold atoms~\cite{Ferre2016,Filinov2012,Saccani_PRL_2012,Dornheim2022,Boninsegni2018} and exotic warm dense quantum plasmas~\cite{Dornheim_PRL_2018,Filinov_PRB_2023,Chuna_PRB_2025}. 

The problem is colloquially referred to as the notorious \emph{analytic-continuation} (AC) problem. The AC problem is known to be ill-posed~\cite{Trefethen2020}. Most practitioners understand the AC as the inverse Laplace transform, which is also known to be ill-posed~\cite{Bertero_RoyalSociety_1982, Epstein2008, Widder2015, shi_CPC_2023, Hadamard1902}. As such dense numerical discretizations of the problem \eqref{eq:two_side_laplace_transform} are ill-conditioned and direct inversion is futile.

To address these challenges and solve for \(f(t)\) from \(F(s)\), approaches have been developed that stabilize the inversion by leveraging the discretization. For example, some approaches start at a low-dimensional grid and increase the dimension until the data is well represented~\cite{Ghanem_PRB_2020} or start at large dense grids and then reduce the grid domain and dimension to improve stability~\cite{bonny_NMR_2020}. 
Other approaches exploit prior knowledge, such as the detailed-balance condition 
\begin{equation}
    f(-t)=f(t) \exp(-\beta t),
\end{equation}
where \(\beta=1/T\) denotes the inverse temperature, to reduce the discretization dimension~\cite{chuna_JPA_2025, Robles2025}. 
The above approaches are motivated by the empirical observation that low-dimensional discretizations are stable, but a thorough analysis of the conditioning of the discretized problem is lacking.

To address this literature gap, we revisit the ill-posedness of the \textit{un-discretized} problem and introduce a weight function, resembling the detailed-balance condition, that mitigates the severity of the Laplace transforms forward-backward amplification observed in Epstein et al.~\cite{Epstein2008}. 
Then, we characterize how the \textit{discretized} AC problem's conditioning is affected by enforcing detailed-balance and changing the number of points in the discretization grid.
We first show that discretizing the two-sided Laplace transform leads to a polynomial regression in the monomial basis, i.e., to a Vandermonde matrix with exponentially distributed nodes. 
We then impose detailed-balance, which transforms the problem into polynomial regression in the Chebyshev basis, i.e., into a Chebyshev Vandermonde-type system with improved conditioning bounds and identify discretization regimes where the condition number remains manageable.
Figure~\ref{fig:resultspreview} previews our main results.
In the low-dimensional regime (\(2 \le n_t \le 30\)), it demonstrates that enforcing the physical constraint known as detailed-balance~\cite{chuna_JPA_2025, Robles2025} improves the discrete inverse problem's conditioning and that low-dimensional discretizations are more stable, justifying sparse kernel-based approaches~\cite{Robles2025, chuna_arXiv_2026}. In the higher-dimensional regime (\(70 \le n_t \le 100\)), the numerical condition numbers reach finite-precision limit and fail to resolve the actual condition numbers, whereas the analytical bounds continue to reveal their super-exponential growth. \\

\begin{figure}[h]
    \centering
    \includegraphics[width=0.89\linewidth]{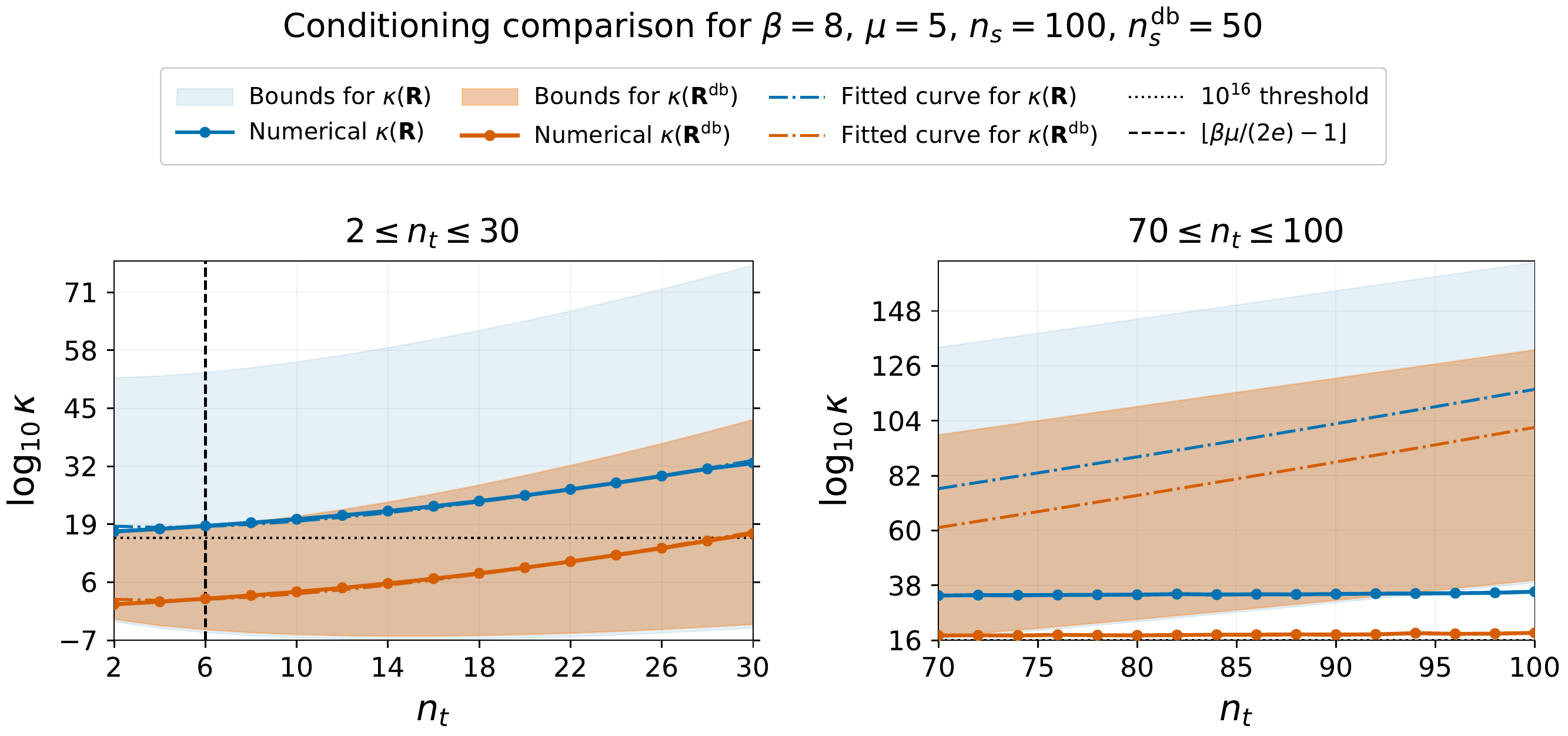}%
    \caption{
    Conditioning of the general and detailed-balanced AC discretizations introduced in Sections~\ref{sec:cond_ac_problem} and~\ref{sec:cond_db_ac_problem} for fixed \(\beta=8\), \(\mu=5\) and \(n_s = 100\). To retain the same grid resolution, the detailed-balance discretization uses \(n_s^{\mathrm{db}} = n_s/2 = 50\) and \(n_t^{\mathrm{db}} = n_t/2\).
    Solid curves show the numerical condition numbers, while the shaded regions indicate the analytical bounds from Theorems~\ref{theo:general_AC_bounds} and~\ref{theo:db_AC_bounds}.
    The dash-dotted curves show super-exponential fits to the numerical condition numbers obtained from the data for \(n_t=2,4,\ldots,30\) and extrapolated to \(n_t \ge 30\).
    The horizontal dotted line marks the numerical threshold \(10^{16}\), which is treated as numerical infinity, while the vertical dashed line at \(n_t = \lfloor\beta\mu/(2e)-1\rfloor\) indicates the boundary of the identified low-dimensional regime.
    }
    \label{fig:resultspreview}
\end{figure}

The paper is organized as follows. Section~\ref{sec:introduction} describes the scientific context in which the inverse two-sided Laplace transform \eqref{eq:two_side_laplace_transform} arises, as well as existing literature on this problem. Section~\ref{sec:two_sided_laplace_transform} establishes the ill-posedness of the continuous problem and introduces a weighted setting, resembling detailed-balance, and determines the associated forward and backward amplification. Section~\ref{sec:cond_ac_problem} presents the lower and upper bounds for the conditioning of the discrete problem. For comparison, Section~\ref{sec:cond_db_ac_problem} considers the detailed-balance variant of the discrete problem, demonstrating its improved conditioning. Finally, Section~\ref{sec:conclusion} summarizes the main findings and discusses possible directions for future work.

\subsection{Mathematical presentation of the analytic continuation problem}
\label{subsec:the_ac_problem}

In theory, problem \eqref{eq:two_side_laplace_transform} can be solved analytically via an explicit inversion formula based on contour integration in the complex plane, dating back to Mellin \cite{Mellin1897, Bromwich1926}
\begin{equation}
    \label{eq:analytic_inversion}
    f(t) 
    = 
    \frac{1}{2\pi i} \lim_{T \rightarrow \infty} \int_{\gamma - iT}^{\gamma + iT} \exp(st) F(s) ds,
\end{equation}
where \(\gamma\) is chosen in the strip of convergence of the two-sided Laplace transform \(F\).
Hence, the inversion formula requires knowledge of \(F\) on the vertical contour
\begin{equation}
    \Gamma_\gamma \coloneqq \{\gamma + iT \, : \, T \in \mathbb R\},
\end{equation}
which extends far beyond the data available on the real interval \(s \in [0,\beta]\). 

From the perspective of the classic Bromwich formula, this means that one would first have to infer the values of \(F\) away from the real interval \(s \in [0,\beta]\), namely on the vertical contour \(\Gamma_\gamma\)
\begin{equation}
    F|_{[0,\beta]} \quad {\leadsto} \quad F|_{\Gamma_\gamma}.
\end{equation}
In the exact, noiseless setting, uniqueness of the analytic continuation follows from analyticity, since values on a real interval with an accumulation point determine the function on the connected component of its domain of analyticity. However, this uniqueness is highly unstable. Small perturbations of the data on \([0, \beta]\) may be amplified dramatically when continuing the function away from the real axis. 

\subsection{Physical origin of the analytic continuation problem}
\label{subsec:the_ac_problem}

We first provide the physical origin of the \emph{AC problem} arising in Monte Carlo simulations of quantum many-body systems.

A broad class of experiments probes the microscopic properties of materials by scattering particles off the sample. The likelihood of scattering is related to the structure of a material, \textit{i.e.}, time and space correlations between particles in the material~\cite{VanHove_PhysRev_1954}. There is a wealth of theory relating the spectral function of these correlation functions to macroscopic material properties, for example, conductivity~\cite{Kubo_RepProgressPhys_1966}. 
These macroscopic material properties are needed to simulate large scale systems.
The time correlation $\tilde{f}$ and its power spectral density $f$ are, by definition, related by the Fourier transform~\cite{pavliotis2014stochastic},
\begin{align}\label{eq:FourierTransform}
    \tilde{f}(z) = \int_{-\infty}^\infty \exp(-izt) f(t) dt \, ,
\end{align}
where $z$ and $t$ are both real. In short, the connections between experimentally measured scattering spectra, microscopic correlation functions, their associated spectral functions, and material properties are a central theme in physics, linking simulations and experiments~\cite{Dornheim_PoP_2023, Lovato_PRC_2015, Landig_NatComm_2015, Carlson_PRC_2002, lovesey_theory_1984, glenzer_XRTSreview}.

Though the typical workflow for a computational physicist is to simulate a system, compute a correlation function $\tilde{f}$, and estimate its spectral function $f$ via Fourier transform, conventional finite-temperature QMC methods do not sample real-time correlations $\tilde{f}$. Instead they access correlation functions defined along the imaginary-time axis, so-called imaginary-time correlation functions (ITCF). This specifies the analytic continuation of $\tilde{f}(z)$ for $z$ into the complex plane as $z \rightarrow -i s$ with $0 < s < \beta$~\cite{Matsubara1955, Jarrell1996}. This analytic continuation is needed so that classic computers can use Markov chain techniques to simulate the quantum system~\cite{Boninsegni2006b}.

The analytic continuation $z \rightarrow -i s$ alters this typical workflow so that instead of estimating the spectrum $f$ via an inverse Fourier transform~\eqref{eq:FourierTransform}, $f$ is estimated via an inverse Laplace transform~\eqref{eq:two_side_laplace_transform}. To see this, substitute $z = -i s$, so that $\exp(-izt) = \exp(-st)$ and $\tilde{f}(-is) = F(s)$. A subtle point, arising naturally from the physics, is that after analytic continuation the ITCF \( F(s)\) diverges outside of $ 0 \le s \le \beta$.~\cite{FetterWalecka_QuantumManyBody_1971}.

In summary, the \emph{AC problem}, (\textit{i.e.}, reconstructing $f$ from finite-interval imaginary-time data via the Laplace transform) is equivalent to analytically continuing the QMC correlation function from imaginary to real time, and then Fourier transforming to obtain the real-frequency spectrum.

\subsection{Related work}
\label{subsec:related_work}

The inversion of the two-sided Laplace transform has a long history. Classic direct numerical inversion methods include those of Schapery~\cite{Schapery1962}, Stehfest~\cite{Stehfest1970}, Talbot~\cite{Talbot1979}, Weeks~\cite{Weeks1966}, and de Hoog et al. \cite{deHoog1982}. 
These methods laid the foundation for much of the numerical inversion literature by providing explicit formulas for approximations of the analytic inverse transform. We refer to Bellman et al. \cite{Bellman1963}, Davies and Martin \cite{Davies1979}, Duffy \cite{Duffy1993}, Cohen \cite{Cohen2007}, and Kuhlman \cite{Kuhlman2013} for reviews and comparative analyses.

Unfortunately, these methods and, in particular, the methods of Schapery and Stehfest operate purely on real-valued arguments, however contour-based approaches require data \(F(s)\) for \(s\) lying in the general complex plane.
Such data are not available in the discretized setting of the \emph{AC problem} considered here. Instead, recent investigations have applied Pad\'e approximations \cite{Vidberg_LowTempPhys_1977, Schott_PRB_2016, Tripolt2019} and moment-based reconstructions \cite{fei_PRL_2021, Filinov_PRB_2023, Dornheim2023}. 

Complementary to direct approaches, many methods instead formulate the inversion as a regularized optimization problem, typically as
\begin{equation}\label{eq:costfxn}
    \|F - \mathcal L f\|^2 + \alpha r(f) \longrightarrow \min_f,
\end{equation}
with \(\alpha > 0\) denoting the regularization parameter and \(r\) the regularization functional. This includes entropic methods \cite{Jarrell1996, Silver1990, BurnierRothkopf_PRL_2013, Boninsegni2018, chuna_JPA_2025}, sparsity-promoting Lasso methods \cite{Otsuki2017}, variational regularization methods~\cite{prokofev_JETP_2013, Han_PRB_2022}, and stochastic sampling methods \cite{Dornheim2022, Ferre2016, Filinov2016, Vitali2010, Mishchenko2000}.
Recently, machine learning methods, in particular neural network-based approaches, have been proposed to learn the mapping from Laplace data to the underlying spectral function in a data-driven manner~\cite{Kades_PRD_2020, Fournier_PRL_2020, shi_CPC_2023}.
Instead of solving \eqref{eq:costfxn}, alternative regularized approaches are built on smearing functions~\cite{Hansen_PRB-GBG_2019, bailas_PTEP_2020, DelDebbio_EuroPhysC_2025}, which pay homage to the seminal Backus--Gilbert method~\cite{Backus1968, Backus1970, Press2007}.

Compared with direct inversion methods, these regularized approaches use prior information about the spectral function \(f\), such as positivity, smoothness, sparsity, moment constraints, or other assumptions. Thus, regularized approaches do not merely approximate the formal analytic inverse, but reformulate the ill-posed inversion as a stabilized reconstruction problem adapted to available data and expected properties of \(f\). A clear example of this is the kernel-based reformulation of \eqref{eq:costfxn}, detailed in \cite{Robles2025}, which remains stable even when data quality deteriorates~\cite{chuna_JCP_2025}.

For convenience, Table~\ref{tab:notation} serves as a reference listing the notation used throughout the paper. \\

\begin{table}[h]
    \centering
    \caption{Notation used throughout the paper.}
    \label{tab:notation}
    \renewcommand{\arraystretch}{1.15}
    \setlength{\tabcolsep}{8pt}
    \small
    \begin{tabular}{|c|l|c|l|}
        \hline
        \(f(t)\)
        & spectral function
        &
        \(F(s)\)
        & Laplace-transformed data
        \\

        \(t\)
        & spectral variable
        &
        \(s\)
        & Laplace variable
        \\

        \(\mu\)
        & spectral truncation parameter
        &
        \(\beta\)
        & inverse temperature
        \\

        \(L^1(\mathbb R,w\,dt)\)
        & weighted signal space
        &
        \(L^1([0,\beta],ds)\)
        & data space
        \\

        \(w(t)\)
        & weight function \(1+\exp(-\beta t)\)
        &
        \(\mathcal L\)
        & two-sided Laplace transform
        \\

        \(\mathcal L^{-1}\)
        & inverse two-sided Laplace transform
        &
        \(\operatorname{Ran}(\mathcal L)\)
        & range of \(\mathcal L\)
        \\

        \hline

        \(\delta_t\)
        & Dirac measure centered at \(t\)
        &
        \(\mathcal F\)
        & family of localized perturbations
        \\

        \(a_w\)
        & forward-backward amplification
        &
        \(R_w\)
        & weighted Rayleigh quotient
        \\

        \hline

        \(n_t\)
        & number of \(t\)-grid intervals
        &
        \(n_s\)
        & number of \(s\)-grid intervals
        \\

        \(h_t\)
        & \(t\)-grid spacing
        &
        \(h_s\)
        & \(s\)-grid spacing
        \\

        \(t_j\)
        & \(j\)th \(t\)-grid node
        &
        \(s_i\)
        & \(i\)th \(s\)-grid node
        \\

        \(\bm\alpha\)
        & reconstruction coefficients
        &
        \(\bm y\)
        & discrete Laplace data
        \\

        \hline

        \(\kappa(\cdot)\)
        & spectral-norm condition number
        &
        \(q\)
        & oversampling factor
        \\

        \(\bm R\)
        & general AC regression matrix
        &
        \(\bm R^{\mathrm{db}}\)
        & detailed-balance regression matrix
        \\

        \(\bm V\)
        & Vandermonde matrix
        &
        \(\bm V^{\mathrm{db}}\)
        & Chebyshev--Vandermonde matrix
        \\

        \(x^j\)
        & \(j\)th monomial
        &
        \(T_j\)
        & \(j\)th Chebyshev polynomial
        \\

        \(x_i\)
        & transformed Vandermonde node
        &
        \(x_i^{\mathrm{db}}\)
        & transformed detailed-balance node
        \\
        
        \hline
    \end{tabular}
\end{table}
\section{The two-sided Laplace transform on a weighted \(L^1\) space}
\label{sec:two_sided_laplace_transform}
In this section, we define the two-sided Laplace transform and analyze the severity of its inverse transform.
We first introduce the forward operator on a suitable weighted \(L^1\) space, providing a setting in which it is well-defined and bounded. 
We then show that its inverse is unbounded, and hence that the inverse problem is ill-posed, with an exponential degree of ill-posedness.
Finally, we quantify the forward-backward amplification of localized perturbations. 
In the unweighted formulation, this amplification can grow exponentially in \(\beta\mu\), see~\eqref{eq:amp_unweighted}, whereas the weighted formulation reduces this growth to at most linear, see~\eqref{eq:amp_weighted}.

In particular, we distinguish between the \emph{degree of ill-posedness}~\cite{Hofmann2010} of the infinite-dimensional inverse problem and the \emph{ill-conditioning} of its finite-dimensional discretizations. 
Accordingly, condition numbers are considered only for the discrete setting in the subsequent Sections~\ref{sec:cond_ac_problem} and \ref{sec:cond_db_ac_problem}.

\subsection{Defining the two-sided Laplace transform}
\label{sec:defining_two_sided_laplace_transform}

Since the kernel \(\exp(-st)\) grows exponentially as \(t \to -\infty\) for \(s>0\), integrability of \(f\) on \(\mathbb R\) alone does not ensure that the two-sided Laplace transform is finite for \(s \in [0,\beta]\). 
We therefore introduce a weighted \(L^1\)-space that compensates for this growth.

\begin{definition}[Two-sided Laplace transform]
    \label{def:two_sided_laplace_transform}
    Let \(\beta > 0\) be fixed and the weight function 
    \begin{equation}
        \label{eq:weight_function}
        w : \mathbb R \longrightarrow (1,\infty), 
        \qquad
        t \longmapsto 1 +  \exp(-\beta t).
    \end{equation}
    The \emph{two-sided Laplace transform} on a compact interval \(s \in [0, \beta]\) is the linear operator
    \begin{equation}
        \mathcal L:
        L^1(\mathbb R, w (t) dt)
        \longrightarrow
        L^1([0,\beta], ds),
        \qquad
        f 
        \longmapsto 
        \mathcal L[f],
    \end{equation}
    defined by
    \begin{equation}
        \label{eq:two_sided_laplace_transform}
        \mathcal L[f](s)
        \coloneqq
        \int_{-\infty}^\infty \exp(-st) f(t) dt,
        \quad s \in [0,\beta].
    \end{equation}
\end{definition}

With this choice of weight, the two-sided Laplace transform defines a bounded linear operator.

\begin{proposition}
    \label{prop:laplace_is_bounded_linear}
    The two-sided Laplace transform \(\mathcal L\) from Definition~\ref{def:two_sided_laplace_transform} is a positive bounded linear operator for all \(\beta > 0\).
\end{proposition}
\begin{proof}
    Let \(\beta > 0\). Positivity follows from the positivity of the kernel \(\exp(-st)\). If \(f \ge 0\) a.e., then \(\exp(-st) f(t) \ge 0\) a.e., and hence \(\mathcal L[f](s) \geq 0\) for all \(s \in [0,\beta]\). Linearity follows from the linearity of the integral. Moreover, since \(\exp(-st) \le w(t)\) for all \(s \in [0,\beta]\) and \(t \in \mathbb R\), Tonelli's theorem yields
    \begin{equation}
        \|\mathcal L[f]\|_{L^1(0,\beta)}
        \le
        \int_0^\beta
        \int_{-\infty}^{\infty}
        \exp(-st)|f(t)| dt ds
        \le
        \int_0^\beta
        \int_{-\infty}^{\infty}
        w(t)|f(t)| dt ds
        =
        \beta \|f\|_{L^1(\mathbb R, w(t) dt)}.
    \end{equation}
    Thus, \(\mathcal L\) is a bounded linear operator from
    \(L^1(\mathbb R, w(t) dt)\) to \(L^1(0, \beta)\) with operator norm \(\|\mathcal L\| \le \beta\).
\end{proof}

The weighted formulation therefore provides a well-defined and bounded forward operator. 
However, boundedness of the forward map does not imply stability of the inverse problem, as we discuss next.

\subsection{Ill-posedness of the inverse two-sided Laplace transform}
\label{sec:ill_posedness}

We next establish the instability of the inverse two-sided Laplace transform and recall its exponential ill-posedness.
Although \(\mathcal L\) is bounded and injective, with injectivity classically following from Lerch~\cite{Lerch1903}, the inverse fails to be well-posed in the sense of Hadamard~\cite{Hadamard1902}. 
Therefore, small perturbations in the data may lead to arbitrarily large perturbations in the reconstruction.

\begin{proposition}
    \label{prop:inv_laplace_transform_unbounded}
    The inverse of the two-sided Laplace transform \(\mathcal L\) from Definition~\ref{def:two_sided_laplace_transform} considered as an operator on its range
    \begin{equation}
        \mathcal L^{-1}:
        L^1((0,\beta), ds) \supset \operatorname{Ran}(\mathcal L) 
        \longrightarrow
        L^1(\mathbb R,w(t)\,dt)
    \end{equation}
    is linear but not bounded for all \(\beta > 0\). Equivalently,
    \(
        \|\mathcal L^{-1}\| = \infty.
    \)
\end{proposition}
\begin{proof}
    To prove this proposition, we present a single example which violates the continuity.
    Let \(\beta > 0\), \(\varphi\in C_c^1(\mathbb R) \setminus \{0\}\) be real-valued and define for sufficiently large \(n \in \mathbb N\)
    \begin{equation}
        f_n(t)
        \coloneqq 
        \frac{
            \sin(nt) \varphi(t)
        }
        {
            \|\sin(n\cdot)\varphi\|_{L^1(\mathbb R,w(t)\,dt)}
        }.
    \end{equation}
    By means of standard estimates and
    integration-by-parts, this sequence satisfies
    \begin{equation}
        \|f_n\|_{L^1(\mathbb R, w(t)\,dt)} = 1,
        \qquad
        \|\mathcal L[f_n]\|_{L^1(0, \beta)} \longrightarrow  0.
    \end{equation}
    Consequently, \(\|\mathcal L^{-1}\| = \infty\) and hence not bounded.
\end{proof}

Although Proposition \ref{prop:inv_laplace_transform_unbounded} gives a concrete example of the instability, the issue is more general.
In fact, for the classic one-sided Laplace transform on \(L^2(0, \infty)\), the Mellin transform yields a generalized singular system \((\sigma_r, u_r, v_r)\) \cite{Epstein2008, Kuhlman2013}, given by
\begin{equation}
    \sigma_r = \sqrt{\frac{\pi}{\cosh(\pi r)}},
    \qquad
    u_r(t) = \exp\left(-i\arg\Gamma(\nicefrac12 + ir)\right) t^{-\nicefrac12 + ir},
    \qquad
    v_r(s) = s^{-\frac12-ir},
    \qquad r \in \mathbb R.
\end{equation}
Since
\begin{equation}
    \sigma_r
    \sim
    \sqrt{2\pi}\,\exp\left(-\frac{\pi}{2}|r|\right)
    \qquad 
    \text{as } |r| \to \infty,
\end{equation}
the singular values \(\sigma_r\) decay exponentially, making the inversion \emph{exponentially ill-posed}.

\subsection{Forward and backward amplification of the two-sided Laplace transform}
\label{sec:forward_backward_amp}

Here, we consider how a Dirac delta perturbation in the \(t\)-variable is amplified through the forward and backward Laplace transform, determined by means of Rayleigh quotients. We use the Dirac delta measure because it isolates the pointwise behavior of the kernel \(\exp(-st)\).
We show that without the weight function \(w(t)\) from \eqref{eq:weight_function} the amplification of a frequency component is exponential, but including the weight function leads to an amplification that is only linear.

Let \(w(t)\) be given as in Definition~\ref{def:two_sided_laplace_transform}. 
We denote by \(M_w^+(\mathbb R)\) the space of nonnegative measures with finite weighted mass
\begin{equation}
    M_w^+(\mathbb R) 
    \coloneqq 
    \{\nu \in M (\mathbb R) \ : \|\nu\|_{M_w^+(\mathbb R)} < \infty\}, 
    \qquad
    \|\nu\|_{M_w^+(\mathbb R)} 
    \coloneqq 
    \int_{\mathbb R} w(t) d \nu(t).
\end{equation}
We extend the Laplace transform to \(\nu \in \mathcal M_w^+(\mathbb R)\) in the natural way by
\begin{equation}
    \mathcal L[\nu](s)
    \coloneqq
    \int_{\mathbb R} \exp(-st)\,d\nu(t).
\end{equation}
Let \(\mathcal F \subset M_w^+(\mathbb R)\) be a family of non-zero measures.
In analogy with the condition number, we capture the forward-backward amplification on \(\mathcal F\) by taking the product of the largest forward amplification and the largest backward amplification on the image \(\mathcal L(\mathcal F)\)
\begin{equation}
    \label{eq:forward_backward_amp}
    a_w(\mathcal F)
    \coloneqq
    \sup_{\nu \in \mathcal F \setminus \{0\}}
    \frac{
        \|\mathcal L[\nu]\|_{L^1(0,\beta)}
    }{
        \|\nu\|_{M_w^+(\mathbb R)}
    }
    \ 
    \sup_{\gamma \in \mathcal L(\mathcal F) \setminus \{0\}}
    \frac{
        \|\mathcal L^{-1} [\gamma]\|_{M_w^+(\mathbb R)}
    }{
        \|\gamma\|_{L^1(0,\beta)}
    }.
\end{equation}
Defining for each \(\nu \in \mathcal F\) the Rayleigh quotient
\begin{equation}
    R_w(\nu)
    \coloneqq
    \frac{
        \|\mathcal L[\nu]\|_{L^1(0,\beta)}
    }{
        \|\nu\|_{M_w^+(\mathbb R)}
    },
\end{equation}
we can write
\begin{equation}
    a_w(\mathcal F)
    =
    \sup_{\nu \in \mathcal F \setminus \{0\}} R_w(\nu)
    \
    \sup_{\nu \in \mathcal F \setminus \{0\}} \frac{1}{R_w(\nu)}
    =
    \frac{
    \sup_{\nu \in \mathcal F \setminus \{0\}} R_w(\nu)
    }{
    \inf_{\nu \in \mathcal F \setminus \{0\}} R_w(\nu)
    }.
\end{equation}
Thus, \(a_w(\mathcal F)\) has two equivalent interpretations. It measures the maximal relative variation of the forward amplification on \(\mathcal F\), and it measures the forward-backward amplification on \(\mathcal F\).

We now use this quantity to examine the effect of the weight function.
\begin{proposition}
    \label{prop:forward_backward_amp}
    Let \(\mathcal L: L^1(\mathbb R, wdt) \longrightarrow L^1([0,\beta],ds)\) denote the weighted two-sided Laplace transform. 
    Let \(\mu>0\) and define the family of Dirac delta measures
    \begin{equation}
        \mathcal F
        \coloneqq
        \{\delta_t \ : \ t\in[-\mu,\mu]\}.
    \end{equation}
    Then, the forward-backward amplification of \(\mathcal L\) on \(\mathcal F\) is
    \begin{equation}
        \label{eq:amplification_weighted}
        a_w({\mathcal F})
        =
        \frac{\nicefrac{\beta\mu}{2}}{\tanh\left(\nicefrac{\beta\mu}{2}\right)}.
    \end{equation}
\end{proposition}
\begin{proof}
    Using the definition of \(w(t)\), the Rayleigh quotient of \(\delta_t\) becomes
    \begin{equation}
        R_w(\delta_t)
        =
        \frac{1}{1+\exp(-\beta t)}
        \int_0^\beta \exp(-st)ds.
    \end{equation}
    For \(t \neq 0\), we have 
    \begin{equation}
        R_w(\delta_t)
        =
        \frac{1}{1 + \exp(-\beta t)}\frac{1-\exp(-\beta t)}{t} = \frac{1}{t} \tanh\left(\frac{\beta t}{2}\right).
    \end{equation}
    Moreover, by continuity
    \begin{equation}
        R_w(\delta_0)
        =
        \lim_{t \to 0} \frac{1}{t} \tanh\left(\frac{\beta t}{2}\right) = \frac{\beta}{2}
    \end{equation}
    The Rayleigh quotient \(R_w(\delta_t)\) attains its maximum at \(t=0\) and minimum at \(|t|=\mu\)
    \begin{equation}
        \sup_{t \in [-\mu,\mu]} R_w(\delta_t) 
        = 
        \frac{\beta}{2},
        \qquad
        \inf_{t \in [-\mu,\mu]} R_w(\delta_t) 
        = 
        \frac{1}{\mu} \tanh\left(\frac{\beta\mu}{2}\right),
    \end{equation}
    delivering the assertion together with \eqref{eq:forward_backward_amp}.
\end{proof}

This proposition shows that the weight \(w\) removes the exponential growth that would otherwise appear in the forward-backward amplification. Indeed, without the weight, the Rayleigh quotient of the same Dirac family is given by
\begin{equation}
    R_1(\delta_t) 
    = 
    \begin{cases}
        \dfrac{1-\exp(-\beta t)}{t}, & t \neq 0, \\
         \beta , & t=0,
    \end{cases}
\end{equation}
yielding
\begin{equation}\label{eq:amp_unweighted}
    a_1(\mathcal F) = \exp(\beta \mu).
\end{equation}
Thus, the unweighted problem yields an exponential forward-backward amplification in \(\beta \mu\). 
Comparing \eqref{eq:amp_unweighted} to the weighted amplification \eqref{eq:amplification_weighted}, which is bounded linearly in \(\beta \mu\) by
\begin{equation}
    \label{eq:amp_weighted}
    a_w(\mathcal F) \le 1 +\frac{\beta \mu}{2} \, ,
\end{equation}
indicates that the weight function greatly reduces the amplification of the transformation.

A similar linear bound appears when, instead of a weight function, the solution is assumed to satisfy detailed-balance. 
Here, one builds the exponential compensation of the weight directly into the model class, as for example in \cite{Robles2025}. 
For the Dirac family, this replaces the \(\delta_t\) in \(\mathcal F\) by its detailed-balance counterpart
\begin{equation}
    \mathcal F_{\mathrm{db}} 
    \coloneqq 
    \{\delta_t + \exp(-\beta t) \delta_{-t} 
    \ : \ 
    t \in [0,\mu]\}.
\end{equation}
The factor \(\exp(-\beta t)\), which otherwise appeared in the weight, is now included into the model itself. 
The forward-backward amplification on \(\mathcal F_{\textbf{db}}\) again grows at most linearly in \(\beta \mu\)
\begin{equation}
    a_1(\mathcal F_{\mathrm{db}}) = a_w(\mathcal F) \le 1 + \frac{\beta \mu}{2}.
\end{equation}
Thus, the impact of the weight function is comparable to the impact of imposing detailed-balance. 
In the subsequent section \ref{sec:disc_db_ac_problem} on the discrete problem, we will only consider the case where the solution is assumed to satisfy detailed-balance.

Comparing the analysis above to that of Epstein \textit{et al}.~\cite{Epstein2008}, Epstein \emph{et al.} already identified a closely related exponential amplification for the one-sided Laplace transform. 
In their filtration analysis, a characteristic-function filter supported on a positive interval produces an exponential factor in the reconstructed signal, whereas shifting the support to a negative interval removes this exponential growth and leads instead to linear attenuation. 
Our analysis complements this observation for the two-sided Laplace transform, where positive and negative \(t\)-values are simultaneously present, and the kernel \(\exp(-st)\) grows exponentially on the negative half-axis. 
The weight function \(w(t)\) compensates for this asymmetry, reducing the forward-backward amplification of the Dirac family \(\mathcal F\) from exponential growth in \(\beta\mu\), to at most linear growth in \(\beta\mu\). Since the weighted amplification still depends on \(\beta\mu\), enlarging the reconstruction domain can nevertheless worsen the inversion, as observed in practice~\cite{chuna_JCP_2025}.

\section{Conditioning of the discrete general AC problem}
\label{sec:cond_ac_problem}

Within this section, we discretize the continuous AC problem in Section~\ref{sec:disc_ac_problem}, delivering a Vandermonde-type linear system. 
Then, in Section~\ref{sec:cond_bounds_ac_problem}, we derive analytical upper and lower bounds for its condition number.
The lower bound establishes super-exponential growth, whereas the upper bound shows that a better conditioning can be achieved when only a small number of reconstruction points is used, see~\eqref{eq:lowdim_upperbound}.
Finally, in Section~\ref{sec:val_cond_bounds_ac_problem}, we numerically validate the derived bounds.

Our conditioning analysis builds on classic results for Vandermonde and Vandermonde-like matrices, in particular norm estimates for inverses and node-separation bounds developed by \emph{Gautschi and others}~\cite{Gautschi1962, Gautschi1975, Gautschi1988, Fasino1992}. 
Since the regression matrix resulting from the discretization of the general AC problem has a specific node distribution together with a diagonal prefactor, we develop estimates tailored to this structure.

\subsection{Discretizing the continuous general AC problem}
\label{sec:disc_ac_problem}
Recalling the physical setting from Section~\ref{subsec:the_ac_problem}, the problem considered here is the reconstruction of the \emph{dynamic structure factor} (DSF) \(f=f(t)\) from the \emph{imaginary-time correlation function} (ITCF) \(F=F(s)\).
This constitutes the \emph{AC problem}, as discussed in Section~\ref{subsec:the_ac_problem}. 
We formalize its \emph{discrete version} below.

Let 
\(
    f \in L^1(\mathbb R, w dt)
\) 
and denote its two-sided Laplace transform by
\(
    F(s) \coloneqq  \mathcal L[f](s),
\)
with
\begin{equation}
    \mathcal L:L^1(\mathbb R, w dt) \longrightarrow L^1([0,\beta], ds).
\end{equation}
The \emph{discrete AC problem} is to reconstruct \(f = f(t)\) from finitely many samples of \(F = F(s)\). 
In particular, we assume that \(F\) is sampled on \(n_s + 1 \in \mathbb N\) equidistant nodes
\begin{equation}
    s_i = i h_s, 
    \qquad 
    i = 0, \ldots, n_s,
    \qquad
    h_s \coloneqq \frac{\beta}{n_s},
\end{equation}
and that the available data is
\begin{equation}
    y_i = F(s_i), 
    \qquad 
    i = 0, \ldots, n_s.
\end{equation}

To obtain a finite-dimensional problem, we restrict the \(t\)-variable to a finite interval \([-\mu, \mu]\) with \(\mu > 0\) and discretize it by \(n_t + 1 \in \mathbb N\) equidistant nodes 
\begin{equation}
    t_j = -\mu + j h_t, 
    \qquad
    h_t = \frac{2\mu}{n_t},
    \qquad
    j=0, \ldots, n_t.
\end{equation}
The integral defining the two-sided Laplace transform~\eqref{eq:two_sided_laplace_transform} is then approximated by a quadrature rule in the \(t\)-variable, yielding
\begin{equation}
    F(s_i) 
    \approx
    \sum_{j=0}^{n_t} w_j \exp(-s_i t_j) f(t_j),
\end{equation}
where \(w_j\) denotes the quadrature weight associated with the node \(t_j\). 
Defining a coefficient vector \(\bm \alpha \in \mathbb R^{n_t+1}\) with
\begin{equation}
    \alpha_j \coloneqq w_j f(t_j),
\end{equation} 
we obtain the discrete linear system
\begin{equation}
    \label{eq:general_AC_reg_mat}
    \bm R \bm \alpha = \bm y, 
    \qquad
    R_{i,j} = \exp(-s_i t_j), 
    \qquad
    \bm R \in \mathbb R^{(n_s + 1)\times(n_t + 1)}.
\end{equation}
Here, \(\bm R\) can be rewritten as 
\begin{equation}
    R_{i,j} = \exp(-s_i(-\mu + j h_t)) = \exp(s_i \mu) \exp(-s_i h_t)^j.
\end{equation}
Introducing the nodes 
\begin{equation}
    x_i \coloneqq \exp(-s_ih_t),
\end{equation}
we obtain
\begin{equation}
    R_{i,j} = \exp(s_i \mu) x_i^j.
\end{equation}
Consequently, the regression matrix admits a factorization into a diagonal matrix and a Vandermonde matrix using the nodes \(x_i\), represented in the monomial basis \(\{x^0, x^1, \ldots, x^{n_t}\}\).

\subsection{Condition number bounds for the discrete general AC problem}
\label{sec:cond_bounds_ac_problem}

We provide explicit lower and upper bounds for the conditioning of the regression matrix \(\bm R\) defined in \eqref{eq:general_AC_reg_mat} in terms of the number of samples \(n_s\), the number of reconstruction points \(n_t\) and the product \(\beta \mu\).

\begin{restatable}{theorem}{generalACbounds}
    \label{theo:general_AC_bounds}
    Let \(\beta, \mu > 0\) and \(n_s, n_t \in \mathbb N\), \(n_s \ge n_t\) be given, and let the regression matrix \(\bm R \in \mathbb R^{(n_s+1) \times (n_t+1)}\) given by
    \begin{equation}
        R_{i,j} = \exp(-s_i t_j),
        \qquad
        s_i = h_s i,
        \qquad
        t_j = -\mu + h_t j,
    \end{equation}
    with \(h_s = \beta / n_s\), \(h_t = 2\mu / n_t\), for \(i = 0, \ldots, n_s\) and \(j = 0, \ldots, n_t\), and set
    \begin{equation}
        q \coloneqq \left\lfloor 
        \frac{n_s + 1}{n_t + 1}
        \right\rfloor.
    \end{equation}
    Then, its condition number satisfies the bounds
    \begin{equation}
        \label{eq:general_AC_bounds}
        \frac{1}{\sqrt{2} n_t^{1/4}} \left(\frac{n_t}{\beta\mu}\right)^{n_t}
        \le
        \kappa(\bm R)
        \le
        \sqrt{\frac{(n_t + 1)(n_s + 1)}{q}} \exp(3\beta\mu)
        \left(
            \frac{2en_s}{\beta \mu q}
        \right)^{n_t}.
    \end{equation}
\end{restatable}
\begin{proof}
    The proof is deferred to Appendix~\ref{app:proof_general_AC_problem}.
\end{proof}

The lower bound is independent of the number of samples \(n_s\). 
Consequently, for fixed \(\beta \mu \), this establishes that the discrete AC problem is super-exponentially ill-conditioned as \(n_t \to \infty\).
In this sense, discretization aggravates the instability, since the inverse two-sided Laplace transform itself is only exponentially ill-posed.
In particular, oversampling in the \(s\)-variable cannot remove this super-exponential growth.
The dependence of the upper bound on \(n_s\) can be isolated through the oversampling factor \(q\). 
Indeed, it holds
\begin{equation} 
    \label{eq:remove_ns_dependence}
    \frac{n_s+1}{q} 
    \le 
    \left(1 + \frac{1}{q}\right) 
    (n_t + 1),
\end{equation}
and therefore
\begin{equation}
    \label{eq:general_AC_upper_bound}
    \kappa(\bm R) 
    \le 
    \sqrt{2} (n_t+1)
    \exp(3\beta\mu) 
    \left(1 + \frac{1}{q}\right)^{n_t+1/2} 
    \left( \frac{2e(n_t+1)}{\beta\mu}\right)^{n_t}.
\end{equation}
Oversampling may improve the actual numerical conditioning, but it does not prevent the super-exponential dependence on the number of reconstruction points \(n_t\).

However, when \(n_t\) stays in the low-dimensional range
\begin{equation}
    1 \le n_t + 1 \le \frac{\beta\mu}{2e},
\end{equation}
we obtain the simpler expression without super-exponential growth in \(n_t\)
\begin{equation}
    \label{eq:lowdim_upperbound}
    \kappa(\bm R)
    \le
    \sqrt{2}
    \left(n_t + 1\right) \exp(3\beta\mu)
    \left(1 + \frac{1}{q}\right)^{n_t+1/2}
\end{equation}
Hence, despite the severe ill-conditioning, the discrete AC problem may remain numerically manageable in this low-dimensional regime. 
This supports the efforts of Chuna et al.~\cite{chuna_arXiv_2026}, who formulate the AC problem as a dictionary learning problem in search of a low-dimensional representation in \(t\). 
But, even in this low-dimensional regime, the upper bound still depends exponentially on the product \(\beta\mu\). 

\subsection{Numerical validation of the condition number bounds for the discrete general AC problem}
\label{sec:val_cond_bounds_ac_problem}

Here, we validate the derived bounds from the previous section by comparing them with the numerically computed condition numbers of \(\bm R\).
In Figure~\ref{fig:general_conditioning}, we compare the numerically computed conditioning of \(\bm R\) from~\eqref{eq:general_AC_reg_mat}, with the analytical bounds derived in Theorem~\ref{theo:general_AC_bounds}.
For the parameter range shown, the numerical condition numbers are consistent with the derived bounds.
The lower bound correctly reflects the rapid growth of the conditioning for each fixed \(\beta\), but does not reproduce the trend as \(\beta\) increases.
In contrast, the upper bound follows the behavior of the condition number more closely, both for fixed \(\beta\) and across increasing values of \(\beta\).
Although neither bound is tight, the analytical estimates make the dependence on \(n_s, n_t\) and \(\beta\mu\) explicit.

\input{figures/general-conditioning}

In the next section, we impose detailed-balance as a \emph{physical constraint} and investigate to what extent it improves the conditioning of the discrete AC problem.

\section{Conditioning of the discrete AC problem with detailed-balance enforced on the solution}
\label{sec:cond_db_ac_problem}

In the previous section, we established that the discretized \emph{general AC problem} is super-exponentially ill conditioned, although its conditioning may remain manageable in a low-dimensional regime. 
In this section, we investigate how imposing detailed-balance as a \emph{physical constraint}, as employed, for example, by Robles et al.~\cite{Robles2025}, affects this conditioning.

Thus, we begin in Section~\ref{sec:disc_db_ac_problem} by incorporating the detailed-balance as a physical constraint into the discretization of the AC problem, again obtaining a Vandermonde-type linear system.
Following that, in Section~\ref{sec:cond_bounds_db_ac_problem}, we derive lower and upper bounds, \eqref{eq:dbVandermondeBound} and \eqref{eq:dbVandermondeUpperBound} respectively, for the condition number. 
While the detailed-balanced problem remains super-exponentially ill-conditioned in general, its conditioning is substantially improved compared with the general AC discretization in the low-dimensional regime, see~\eqref{eq:db_lowdim_upperbound}.
Finally, in Section~\ref{sec:val_cond_bounds_db_ac_problem}, we numerically validate the derived bounds.

\subsection{Discretizing the continuous detailed-balanced AC problem}
\label{sec:disc_db_ac_problem}
We now consider the discrete AC problem under the additional assumption that the DSF satisfies the detailed-balance symmetry. 
This symmetry stems from the Kubo-Martin-Schwinger condition for equilibrium correlation functions at inverse temperature \(\beta\)~\cite{FetterWalecka_QuantumManyBody_1971}.
The key structural observation is that imposing this symmetry relation transforms the discretized problem into polynomial regression in the Chebyshev basis.

Specifically, we assume that  \(f \in L^1(\mathbb R, dt)\) satisfies
\begin{equation}
    f(-t) = \exp(-\beta t) f(t),
    \qquad
    t \ge 0.
\end{equation}
For the two-sided Laplace transform \(\mathcal L : L^1(\mathbb R, w dt) \longrightarrow L^1([0,\beta], ds)\), the corresponding imaginary-time correlation function \(F(s) \coloneqq \mathcal L[f](s)\) then satisfies the symmetry~\cite{Dornheim_POP_2023_2}
\begin{equation}
    F(s) = F(\beta - s), 
    \qquad
    s \in [0, \beta].
\end{equation}
Thus, it is sufficient to sample \(F\) on the half interval \([0,\beta/2]\). Accordingly, let \(n_s \in \mathbb N\), and define the equidistant imaginary-time nodes
\begin{equation}
    s_i = i h_s, 
    \qquad
    h_s \coloneqq \frac{\beta}{2 n_s},
    \qquad
    i = 0, \ldots, n_s.
\end{equation}
The available data is given by 
\begin{equation}
    y_i = F(s_i),
    \qquad
    i = 0, \ldots, n_s.
\end{equation}
To obtain a finite-dimensional problem, we restrict the positive \(t\)-variable to a finite interval \([0,\mu]\) with \(\mu > 0\) and discretize it by \(n_t + 1 \in \mathbb N\) equidistant nodes
\begin{equation}
    t_j = j h_t,
    \qquad
    h_t \coloneqq \frac{\mu}{n_t},
    \qquad
    j = 0, \ldots, n_t.
\end{equation}
Using the detailed-balance, the two-sided Laplace transform can be written as an integral over the positive half-axis only. For \(s \in [0, \beta]\), we thus obtain
\begin{equation}
    F(s) 
    = 
    \int_{-\infty}^{\infty} \exp(-st) f(t) dt
    = 
    \int_0^\infty (\exp(-st) + \exp(-(\beta-s)t)) f(t) dt
\end{equation}
Approximating this integral by a quadrature rule on \([0, \mu]\) yields
\begin{equation}
    F(s_i) \approx \sum_{j=0}^{n_t} w_j (\exp(-s_i t_j) + \exp(-(\beta-s_i)t_j)) f(t_j),
\end{equation}
where \(w_j\) denotes the quadrature weight associated with the nodes \(t_j\). Defining the coefficient vector \(\bm \alpha \in \mathbb R^{n_t+1}\) by
\begin{equation}
    \alpha_j \coloneqq w_j f(t_j),
    \qquad
    j= 0, \ldots, n_t,
\end{equation}
we obtain the discrete linear system
\begin{equation}
    \label{eq:db_AC_reg_mat}
    \bm R^{\mathrm{db}} \bm \alpha = \bm y,
    \qquad
    R^{db}_{i,j} = \exp(-s_i t_j) + \exp(-(\beta-s_i)t_j),
    \qquad
    \bm R^{\mathrm{db}} \in \mathbb R^{(n_s+1)\times (n_t+1)}.
\end{equation}
The entries of the detailed-balanced regression matrix can be rewritten as
\begin{equation}
    R_{i,j}^{\mathrm{db}} 
    = 
    \exp\left(
        -\frac{\beta}{2} \, t_j
    \right)
    \left(
        \exp\left(
            \left(\frac{\beta}{2}-s_i\right) 
            t_j
        \right)
        + 
        \exp\left(
            -\left(\frac{\beta}{2} - s_i\right)
            t_j
        \right)
    \right),
\end{equation}
which yields the hyperbolic representation
\begin{equation}
    R_{i,j}^{\mathrm{db}} 
    = 
    2 \exp \left(
        -\frac{\beta}{2}\,t_j
    \right) 
    \cosh \left(
        \left(\frac{\beta}{2}-s_i\right) t_j
    \right).
\end{equation}
Since \(t_j = j h_t\), the identity
\begin{equation}
    \cosh(jx) = T_j(\cosh x),
\end{equation}
where \(T_j\) denotes the \(j\)th Chebyshev polynomial of the first kind, gives
\begin{equation}
    R_{i,j}^{\mathrm{db}} 
    = 
    2\exp\left(-\frac{\beta}{2} \, t_j\right) T_j\left(\cosh\left(\left(\frac{\beta}{2} \, - s_i\right) h_t\right)\right).
\end{equation}
Introducing the nodes
\begin{equation}
    x_i^\mathrm{db} 
    \coloneqq
    \cosh\left(\left(\frac{\beta}{2} - s_i\right) h_t\right),
\end{equation}
\(\bm R^\mathrm{db}\) can be rewritten as 
\begin{equation}
    R_{i,j}^\mathrm{db} 
    = 
    2 \exp\left(-\frac{\beta}{2} \, t_j\right) T_j(x_i^\mathrm{db}).
\end{equation}
Consequently, the regression matrix admits a factorization into a diagonal matrix and a Vandermonde-type matrix using the nodes \(x_i^\mathrm{db}\), represented in the Chebyshev basis \(\{T_0, T_1, \ldots, T_{n_t}\}\).

\subsection{Condition number bounds for the discrete detailed-balanced AC problem}
\label{sec:cond_bounds_db_ac_problem}

We provide explicit lower and upper bounds for the conditioning of the regression matrix \(\bm R^{\mathrm{db}}\) defined in \eqref{eq:db_AC_reg_mat} in terms of the number of samples \(n_s\), the number of reconstruction points \(n_t\) and the product \(\beta\mu\).
Moreover, since row permutations do not affect the condition number, we may write the nodes as 
\(
    x_i^\mathrm{db} = \cosh(s_i h_t),
\)
leading to the following result.

\begin{restatable}{theorem}{dbACbounds}
    \label{theo:db_AC_bounds}
    Let \(\beta, \mu > 0\) and \(n_s,n_t \in \mathbb N\), \(n_s \ge n_t\) be given and let the regression matrix \(\bm R^{\mathrm{db}} \in \mathbb R^{(n_s+1) \times (n_t+1)}\) given by
    \begin{equation}
        R_{i,j}^{\mathrm{db}}
        = 
        2 
        \exp\left(-\frac{\beta}{2} \, t_j\right) 
        T_j\left(\cosh(s_ih_t)\right),
        \qquad
        s_i = h_s i,
        \qquad
        t_j = h_t j,
    \end{equation}
    with \(h_s = \beta / (2n_s)\), \(h_t = \mu / n_t\), for \(i = 0, \ldots, n_s\) and \(j = 0, \ldots, n_t\), and set
    \begin{equation}
        q \coloneqq \left\lfloor 
        \frac{n_s + 1}{n_t + 1}
        \right\rfloor.
    \end{equation}
    Then, its condition number satisfies the bounds
    \begin{equation}
        \label{eq:db_AC_bounds}
        2 \left(
            \frac{
                2n_t
            }{
                \beta\mu
            }
        \right)^{2n_t}
        \le 
        \kappa(\bm R^{\mathrm{db}})
        \le
        \sqrt{
            \frac{
                (n_s+1)(n_t+1)
            }{
                q
            }
        }
        \exp(\beta\mu)
        \left(
            \frac{
                4e n_s
            }{
                \beta\mu q
            }
        \right)^{2n_t}.
    \end{equation}
\end{restatable}
\begin{proof}
    The proof is deferred to Appendix~\ref{app:db_AC_bounds}.
\end{proof}

To compare the discretizations of the detailed-balanced with the general AC problem at the same resolution, we choose the sample sizes for the detailed-balance AC problem as
\begin{equation}
    n_s \mapsto \frac{n_s}{2} \in \mathbb N,
    \qquad
    n_t \mapsto \frac{n_t}{2} \in \mathbb N.
\end{equation}
Indeed, detailed-balance reduces the \(s\)-domain to \([0, \beta/2]\), while the negative part of the \(t\)-domain is determined by reflection from its positive half.
Thus, both domains are half as long as in the general AC problem, so we halve the corresponding sample sizes to keep the same resolution.

As a consequence, for this choice, the condition number of \(\bm R^{\mathrm{db}}\) satisfies the lower bound
\begin{equation}
    \label{eq:dbVandermondeBound}
    \kappa(\bm R^{\mathrm{db}})
    \ge
    2
    \left(
        \frac{n_t}{\beta \mu}
    \right)^{n_t}.
\end{equation} 
For fixed \(\beta \mu\), this shows that the detailed-balanced discrete AC problem remains super-exponentially ill-conditioned.
As in the general AC problem, oversampling in the \(s\)-variable cannot remove this super-exponential growth. 
The dependence of the upper bound on \(n_s\) can likewise be isolated as in the general AC problem, via using \eqref{eq:remove_ns_dependence}, yielding
\begin{equation}
    \label{eq:dbVandermondeUpperBound}
    \kappa(\bm R^{\mathrm{db}})
    \le
    \sqrt{2}
    \left(\frac{n_t}{2} + 1\right)
    \exp(\beta\mu)
    \left(1 + \frac{1}{q}\right)^{n_t+1/2}
    \left(
        \frac{2e (n_t+1)}{\beta \mu}
    \right)^{n_t}.
\end{equation}
Comparing with the general AC problem, imposing detailed-balance nearly halves the prefactor from \(n_t+ 1\) to \(n_t/2 + 1\) and the exponential rate in \(\beta\mu\) from \(3\beta\mu\) to \(\beta\mu\). Equivalently, the corresponding exponential base is reduced from
\begin{equation}
    e^3\approx 20.09
    \qquad\text{to}\qquad
    e^{1}\approx 2.71.
\end{equation}

When the number of nodes remains in the same low-dimensional range as for the regression matrix \(\bm R\) of the discretized general AC problem
\begin{equation}
    1 \le n_t + 1 \le \frac{\beta\mu}{2e},
\end{equation}
we obtain a \emph{substantially improved conditioning}
\begin{equation}
    \label{eq:db_lowdim_upperbound}
    \kappa(\bm R^{\mathrm{db}})
    \le
    \sqrt{2}
    \left(\frac{n_t}{2} + 1\right)
    \exp(\beta\mu)
    \left(1 + \frac{1}{q}\right)^{n_t+1/2}
\end{equation}
Comparing \eqref{eq:dbVandermondeBound} with the lower bound in \eqref{eq:general_AC_bounds}, it is evident that imposing detailed-balance does not remove the super-exponential ill-conditioning. At the same time, comparing \eqref{eq:db_lowdim_upperbound} with \eqref{eq:lowdim_upperbound}, shows that it substantially improves the conditioning in the low-dimensional regime and enlarges the range of numerically manageable discretizations.

\subsection{Numerical validation of the condition number bounds for the discrete detailed-balanced AC problem}
\label{sec:val_cond_bounds_db_ac_problem}

Here, we validate the derived bounds from the previous section by comparing them with the numerically computed condition numbers of \(\bm R^{\mathrm{db}}\). In Figure~\ref{fig:db_conditioning}, we compare the numerically computed condition numbers of \(\bm R^{\mathrm{db}}\) with the analytical bounds derived in Theorem~\ref{theo:db_AC_bounds}.
Most notably, the comparison with Figure~\ref{fig:general_conditioning} shows that imposing detailed-balance substantially improves the numerical conditioning of the discrete AC problem. 
In particular, much larger discretizations remain below numerical infinity.

For the parameter range shown, the numerical condition numbers are again consistent with the derived bounds. 
Similar to the general AC problem, the lower bound captures the rapid growth of the conditioning for each fixed \(\beta\), but does not reproduce the dependence on \(\beta\).
The upper bound follows the numerical behavior again more closely. 
While the detailed-balanced problem remains ill-conditioned as the discretization size increases, the numerical results confirm the substantial improvement in terms of conditioning over the general AC discretization predicted by the analysis.

\input{figures/db-conditioning}

Having analyzed and numerically validated the conditioning of both discretizations of the AC problem, we conclude the paper by summarizing the main findings and their implications.

\section{Conclusion}
\label{sec:conclusion}

In this work, we considered the ill-posedness of the inverse two-sided Laplace transform (see Definition~\ref{def:two_sided_laplace_transform}) and analyzed the ill-conditioning of its finite-dimensional discretizations occurring in quantum Monte Carlo simulations of quantum many-body systems.

In Section~\ref{sec:two_sided_laplace_transform}, we demonstrated that the inverse two-sided Laplace transform is an unbounded operator in Proposition~\ref{prop:inv_laplace_transform_unbounded} and recalled the exponential decay of its singular values.
Moreover, we quantified the forward-backward amplification of localized perturbations. 
While the unweighted formulation admits exponential amplification in \(\beta\mu\), the weighted formulation, using the physically motivated weight function
\begin{equation}
    w(t) = 1 + \exp(-\beta t),
\end{equation}
reduces the amplification to at most linear in \(\beta\mu\), see Proposition~\ref{prop:forward_backward_amp}.
In particular, imposing detailed-balance directly for the model class equivalently accounts for the compensating effect of the weight function.

Next, we investigated the finite-dimensional AC problem with and without imposing detailed-balance in Section~\ref{sec:cond_ac_problem} and Section~\ref{sec:cond_db_ac_problem}, respectively.

First, we discretized the general AC problem using equidistant grids. 
The resulting regression matrix~\eqref{eq:general_AC_reg_mat} takes the form of a monomial Vandermonde matrix with exponentially distributed nodes and a diagonal prefactor. 
We derived lower and upper bounds for its condition number in Theorem~\ref{theo:general_AC_bounds}, showing that it grows super-exponentially with the number of points \(n_t\) in the \(t\)-variable for fixed \(\beta\mu\). 
Increasing the number of samples \(n_s\) in the \(s\)-variable cannot compensate for this growth. 
Via the upper bound, we identified a low-dimensional regime in which the discrete problem can remain numerically manageable.

Second, we discretized the detailed-balanced AC problem using equidistant grids. 
Here, imposing detailed-balance as a physical constraint changes the structure of the regression matrix~\eqref{eq:db_AC_reg_mat} to a Chebyshev Vandermonde-type matrix, again with exponentially distributed nodes and a diagonal prefactor.
We derived lower and upper bounds for its condition number in Theorem~\ref{theo:db_AC_bounds}, which likewise establish super-exponential growth with increasing number of samples \(n_t\) in the \(t\)-variable.

At the same resolution, however, detailed-balance substantially improves the upper bound. 
Comparing~\eqref{eq:general_AC_bounds} and~\eqref{eq:db_AC_bounds}, the exponential dependence on \(\beta\mu\) is reduced from \(\exp(3\beta\mu)\) to \(\exp(\beta\mu)\). 
This improves the corresponding exponential base from \(e^3\approx20.09\) to \(e\approx2.71\), which particularly helps in the low-dimensional regime. Our numerical experiments in Section~\ref{sec:val_cond_bounds_ac_problem} and \ref{sec:val_cond_bounds_db_ac_problem} confirmed that the aforementioned improvement applies also for the numerically computed condition numbers.
Thus, imposing detailed-balance does not remove the ill-conditioning but it strongly mitigates its severity.

Although our analysis is formulated for a Dirac-delta representation of the unknown spectrum, its implications extend to a broader class of kernel-based models. 
When the basis functions are generated by translating a common kernel, \(\phi_j(t) = \phi(t-t_j)\), the two-sided Laplace transform factorizes as
\begin{equation}
    \mathcal L[\phi_j] (s) = \exp(-st_j) \mathcal L[\phi](s).
\end{equation}
Hence, replacing Dirac-delta kernels by translated kernels changes the regression matrix only by an additional diagonal row scaling, and the present conditioning analysis applies up to the condition number of this prefactor.
This includes, for example, Gaussian kernels considered by Robles \emph{et al.}~\cite{Robles2025}.
Although the additional diagonal scaling may worsen the conditioning, kernel-based representations can require substantially fewer basis functions to resolve smooth spectra, hence keeping the reconstruction within the better-conditioned low-dimensional regime identified above.

Overall, our results indicate that practical analytic continuation should be performed in a low-dimensional regime in the \(t\)-variable, while oversampling in the \(s\)-variable does not substantially improve the conditioning of the discrete AC problem. 
In this sense, restricting the reconstruction to a suitably low-dimensional discretization can itself be viewed as a form of \emph{self-regularization}~\cite{natterer1977regularisierung,pereverzev2000characterization,mathe2001optimal}.

Our analysis further demonstrates the benefit of incorporating available physical constraints directly into the model. 
For the DSF considered here, the detailed-balance relation implied by the Kubo--Martin--Schwinger condition reduces the admissible solution space and substantially improves the conditioning. 
Problem-specific physical constraints or prior information are also available in a variety of related quantum analytic-continuation problems, including the reconstruction of single-particle spectral functions from imaginary-time Green's functions~\cite{Jarrell1996,Otsuki2017}, dynamical spin response functions from imaginary-time correlation functions~\cite{Sandvik_PRB_1998, ShaoSandvik_PhysRep_2023}, optical conductivities from current--current correlation functions~\cite{Rabani2002}, and hadronic spectral functions from Euclidean correlators in lattice QCD~\cite{Asakawa_PPNP_2001}.

A systematic investigation of how such problem-specific physical constraints affect the conditioning of the discretized inverse problem is an interesting direction for future work.

\section*{Acknowledgements}
We are grateful for stimulating discussions with Uwe Hernandez Acosta and Alexander Benedix Robles.

This work has received funding from the European Research Council (ERC) under the European Union’s Horizon 2022 research and innovation programme (Grant agreement No. 101076233, "PREXTREME"). 
Views and opinions expressed are, however, those of the authors only and do not necessarily reflect those of the European Union or the European Research Council Executive Agency. Neither the European Union nor the granting authority can be held responsible for them.
We gratefully acknowledge funding from the Deutsche Forschungsgemeinschaft (DFG) via project DO 2670/1-1.

This work was partially funded by the Center for Advanced Systems Understanding (CASUS), financed by Germany's Federal Ministry of Research, Technology and Space (BMFTR) and by the Saxon Ministry for Science, Culture and Tourism (SMWK) with tax funds on the basis of the budget approved by the Saxon State Parliament.

\newpage

\appendix
\section{Proof: Condition number bounds for the general AC problem}
\label{app:proof_general_AC_problem}
\generalACbounds*
\begin{proof}
    Let the Gram matrix 
    \begin{equation}
        \bm G \coloneqq \bm R^\top \bm R.
    \end{equation}
    Then, the conditioning of \(\bm R\) is determined by 
    \begin{equation}
        \kappa(\bm R) 
        = 
        \sqrt{\frac{\max \lambda (\bm G)}{\min\lambda(\bm G)}}.
    \end{equation}
    
    We begin with proving the lower bound first by considering the polynomial
    \begin{equation}
        (1-x)^{n_t} = \sum_{j=0}^{n_t} c_j x^j, \qquad c_j \coloneqq (-1)^j {n_t \choose j}.
    \end{equation}
    Its Euclidean norm satisfies
    \begin{equation}
        \|\bm c\|_2^2 = \sum_{j=0}^{n_t} {n_t \choose j}^2 = {2 n_t \choose n_t},
    \end{equation}
    while
    \begin{equation}
        (\bm R \bm c)_i = \exp(\mu s_i) \left(1-\exp(-h_t s_i)\right)^{n_t}.
    \end{equation}
    Using \(1 - \exp(-x) \le x\) for \(x \ge 0\) and \(s_i \le \beta\), we obtain
    \begin{equation}
        \min \lambda(\bm G) 
        \le 
        \frac{
            \|\bm R \bm c\|_2^2
        }{
            \|\bm c\|_2^2
        } 
        \le 
        \frac{1}{{2 n_t \choose n_t}} (\beta h_t)^{2n_t} 
        \sum_{i=0}^{n_s} \exp(2\mu s_i).
    \end{equation}
    
    On the other hand, evaluating the Rayleigh quotient at the first canonical
    basis vector \(\bm e_0\) yields
    \begin{equation}
        \max \lambda(\bm G) \ge \|\bm R \bm e_0\|_2^2 = \sum_{i=0}^{n_s} \exp(2\mu s_i).
    \end{equation}
    Consequently,
    \begin{equation}
        \kappa(\bm R)^2
        \ge
        {2n_t \choose n_t}(\beta h_t)^{-2n_t}.
    \end{equation}
    Using standard estimates, we obtain
    \begin{equation}
        {2 n_t \choose n_t} \ge \frac{4^{n_t}}{2\sqrt{n_t}},
    \end{equation}
    and substituting \(h_t=2\mu/n_t\) delivers the asserted lower bound
    \begin{equation}
        \kappa(\bm R)^2 
        \ge
        \frac{1}{2\sqrt{n_t}} 
        \left(\frac{n_t}{\beta \mu}\right)^{2n_t}.
    \end{equation}
    
    We proceed by proving the upper bound. Recall that
    \begin{equation}
        q 
        \coloneqq 
        \left\lfloor \frac{n_s + 1}{n_t + 1}\right\rfloor.
    \end{equation}
    Consider the \(q\) disjoint sets of row indices
    \begin{equation}
        i_{r,k} = r + k q, 
        \qquad
        r = 0, \ldots, q - 1,
        \qquad
        k = 0, \ldots, n_t,
    \end{equation}
    and denote the corresponding square submatrices of \(\bm R\) by \(\bm B_r\).
    With \(\bm C\) containing the remaining rows, then
    \begin{equation}
        \bm G 
        = 
        \sum_{r=0}^{q-1} \bm B_r^\top \bm B_r + \bm C^\top \bm C.
    \end{equation}
    Since \(\bm C^\top \bm C\) is positive semidefinite, it follows
    \begin{equation}
        \label{eq:lower_min_lambda_G}
        \min \lambda(\bm G) \ge \sum_{r=0}^{q-1} \min \lambda (\bm B_r^\top \bm B_r).
    \end{equation}
    Each block admits the factorization
    \begin{equation}
        (\bm B_r)_{k,j} 
        =
        \exp(\mu s_{i_{r,k}}) \exp\left(-h_t s_{i_{r,k}}\right)^j,
    \end{equation}
    for \(k, j = 0, \ldots, n_t\). Since \(\exp(\mu s_{i_{r,k}}) \ge 1\), defining \((\bm V_r)_{k,j} \coloneqq x_{r,k}^j\) with \(x_{r,k} \coloneqq \exp(-h_t s_{i_{r,k}})\) yields
    \begin{equation}
        \|\bm B_r z\|_2 \ge \|\bm V_r z\|_2,
        \qquad
        z \in \mathbb C^{n_t+1},
    \end{equation}
    and therefore
    \begin{equation}
        \min \lambda(\bm B_r^\top \bm B_r) 
        \ge 
        \min \lambda(\bm V_r^\top \bm V_r) 
        = 
        \frac{1}{\|\bm V_r^{-1}\|_2^2}.
    \end{equation}
    
    To bound \(\|\bm V_r^{-1}\|_2\), we use Proposition~\ref{prop:cheb_vander_inverse_estimate}
    \begin{equation}
        \|\bm V_r^{-1}\|_2 \le \sum_{k=0}^{n_t} \prod_{l=0, l \neq k}^{n_t} \frac{1 + x_{r,l}}{|x_{r,k} - x_{r,l}|}.
    \end{equation}
    Since \(0 < x_{r,l} \le 1\)
    \begin{equation}
        \prod_{l=0, l \neq k}^{n_t} (1 + x_{r,l}) 
        \le 
        2^{n_t}
    \end{equation}
    We proceed by estimating the denominator. Consider, for \(l > k\)
    \begin{equation}
        x_{r,k} - x_{r,l}
        = 
        \exp(-h_t s_{i_{r,k}}) (1 - \exp(-h_t (s_{i_{r,l}} - s_{i_{r,k}})).
    \end{equation}
    Using \(1-\exp(-x) \ge x\exp(-x),\) for \(x \ge 0\), we obtain
    \begin{equation}
        x_{r,k} - x_{r,l}
        \ge
        \exp(-h_t s_{i_{r,l}})
        h_t (s_{i_{r,l}} - s_{i_{r,k}}), 
    \end{equation}
    which can be rewritten as
    \begin{equation}
        x_{r,k} - x_{r,l}
        \ge
        h_t h_s
        \exp(-h_t s_{i_{r,l}}) 
        (i_{r,l} - i_{r,k})
    \end{equation}
    and by using \(s_{i_{r,l}} \le \beta\), we obtain
    \begin{equation}
        x_{r,k} - x_{r,l}
        \ge
        \underbrace{q h_t h_s\exp(-h_t \beta)}_{=: \delta}
            (l - k).
    \end{equation}
    Hence, we have
    \begin{equation}
        |x_{r,k} - x_{r,l}|
        \ge
        \delta |k-l|,
    \end{equation}
    which yields
    \begin{equation}
        \prod_{l \neq k} |x_{r,k} - x_{r,l}|
        \ge
        \delta^{n_t} 
        \prod_{l \neq k} |k-l|.
    \end{equation}
    This can be further simplified to
    \begin{equation}
        \prod_{l \neq k} |x_{r,k} - x_{r,l}|
        \ge
        \delta^{n_t}
        \big(k(k-1) \cdots 1\big)
        \big(1\cdot2\cdots(n_t-k)\big)
        =
        \delta^{n_t}k!(n_t-k)!.
    \end{equation}
    Consequently
    \begin{equation}
        \|\bm V_r^{-1}\|_2
        \le
        \frac{2^{n_t}}{\delta^{n_t}} \sum_{k=0}^{n_t} \frac{1}{k! (n_t-k)!}
        =
        \frac{2^{n_t}}{\delta^{n_t}} \sum_{k=0}^{n_t} {n_t \choose k}
        =
        \frac{4^{n_t}}{n_t! \delta^{n_t}}.
    \end{equation}
    Inserting this into \eqref{eq:lower_min_lambda_G} gives the lower bound
    \begin{equation}
        \label{eq:lower_bound_min_G}
        \min \lambda(\bm G)
        \ge 
        q (n_t!)^2 \exp(-4\beta\mu) \left(\frac{\beta \mu q}{2 n_s n_t}\right)^{2n_t}.
    \end{equation}
    
    For the largest eigenvalue of \(\bm G\), we use
    \begin{equation}
        \max \lambda (\bm G)
        =
        \|\bm R\|_2^2
        \le
        \|\bm R\|_F^2.
    \end{equation}
    Since \(0 \le s_i \le \beta\) and \(-\mu \le t_j \le \mu\)
    \begin{equation}
        |R_{i,j}|^2 = \exp(-2s_it_j) \le \exp(2\beta\mu),
    \end{equation}
    and hence
    \begin{equation}
        \max \lambda(\bm G) \le (n_s+1)(n_t+1)\exp(2\beta\mu).
    \end{equation}
    
    Combining this together with \eqref{eq:lower_bound_min_G} yields
    \begin{equation}
        \kappa(\bm R)^2 
        = 
        \frac{(n_s+1)(n_t+1)}{q} 
        \exp(6 \beta \mu) 
        \frac{1}{(n_t!)^2} 
        \left(\frac{
            2 n_s n_t
            }{
            \beta\mu q
            }\right)^{2n_t}.
    \end{equation}
    Using the fact that 
    \begin{equation}
        n_t! \ge \left(\frac{n_t}{e}\right)^{n_t},
    \end{equation}
    we finally get the asserted upper bound
    \begin{equation}
        \kappa(\bm R)
        \le
        \sqrt{\frac{(n_s+1)(n_t+1)}{q}} 
        \exp(3\beta\mu) 
        \left(\frac{
            2 e n_s
            }{
            \beta\mu q
            }\right)^{n_t}.
    \end{equation}
\end{proof}

\section{Proof: Condition number bounds for the detailed-balanced AC problem}
\label{app:db_AC_bounds}
\dbACbounds*
\begin{proof}
    For brevity, throughout this proof we write \(\bm R \coloneqq \bm R^{\mathrm{db}}\).
    Let the Gram matrix 
    \begin{equation}
        \bm G \coloneqq \bm R^\top \bm R.
    \end{equation}
    Then, the conditioning of \(\bm R\) is determined by
    \begin{equation}
        \kappa(\bm R) 
        = 
        \sqrt{\frac{\max \lambda (\bm G)}{\min\lambda(\bm G)}}.
    \end{equation}
    
    We begin with proving the lower bound first by considering the polynomial
    \begin{equation}
        (1-x)^{n_t} 
        = 
        \sum_{j=0}^{n_t} c_j T_j(x).
    \end{equation}
    Since the leading coefficient of \(T_{n_t}\) in the monomial basis is \(2^{n_t-1}\), comparing both sides gives \begin{equation}
        c_{n_t} = (-1)^{n_t}2^{1-n_t}.
    \end{equation}
    Recalling the diagonal prefactor \(\exp(-\beta t_j/2)\) in the columns of \(\bm R\), we define
    \begin{equation}
        \widetilde c_j
        \coloneqq 
        \exp\left(\frac{\beta}{2} t_j\right) c_j,
    \end{equation}
    which yields
    \begin{equation}
        \|\widetilde{\bm{c}}\,\|_2
        \ge
        \exp\left(
            \frac{\beta}{2}t_{n_t}
        \right)|c_{n_t}| 
        = 
        2^{1-n_t} 
        \exp\left(\frac{\beta\mu}{2}\right).
    \end{equation}
    Further, it holds
    \begin{equation}
        |(\bm R \widetilde{\bm{c}})_i|
        = 
        2|1 - \cosh(s_i h_t)|^{n_t}
        \le
        2|1-\cosh(n_sh_sh_t)|^{n_t}.
    \end{equation}
    Using \(\cosh x - 1 = 2 \sinh^2 (x/2)\) and \(\sinh(x) \le x \exp(x)\), we obtain
    \begin{equation}
        |(\bm R \widetilde{\bm{c}})_i|
        \le
        2^{n_t+1} \sinh^{2 n_t}\left(
            \frac{\beta \mu}{4 n_t}
        \right)
        \le
        2^{n_t+1}\left(\frac{\beta\mu}{4 n_t}\right)^{2n_t}
        \exp\left(
            \frac{\beta\mu}{2}
        \right).
    \end{equation}
    Thus, we obtain
    \begin{equation}
        \min \lambda(\bm G) 
        \le 
        \frac{
            \|\bm R \widetilde{\bm{c}}\,\|_2^2
        }{
            \|\widetilde{\bm{c}}\,\|_2^2
        } 
        \le 
        \frac{
            (n_s+1) 2^{2n_t+2}
            \left(
                \frac{\beta\mu}{4n_t}
            \right)^{4n_t} \exp(\beta\mu)
        }{
            2^{2-2n_t} \exp(\beta\mu)
        }
        \le
        (n_s+1) \left(
            \frac{\beta\mu}{2n_t}
        \right)^{4n_t}
    \end{equation}
    
    On the other hand, evaluating the Rayleigh quotient at the first canonical
    basis vector \(\bm e_0\) yields
    \begin{equation}
        \max \lambda(\bm G) 
        \ge 
        \|\bm R \bm e_0\|_2^2 
        = 
        4 (n_s+1).
    \end{equation}
    
    Consequently, delivering the asserted lower bound
    \begin{equation}
        \kappa(\bm R)^2
        \ge
        4 \left(
        \frac{2n_t}{\beta\mu}
        \right)^{4n_t}.
    \end{equation}
    
    We proceed by proving the upper bound. Recall that
    \begin{equation}
        q 
        \coloneqq 
        \left\lfloor \frac{n_s + 1}{n_t + 1}\right\rfloor.
    \end{equation}
    Consider the \(q\) disjoint sets of row indices
    \begin{equation}
        i_{r,k} = r + k q, 
        \qquad
        r = 0, \ldots, q - 1,
        \qquad
        k = 0, \ldots, n_t,
    \end{equation}
    and denote the corresponding square submatrices of \(\bm R\) by \(\bm B_r\).
    With \(\bm C\) containing the remaining rows, then
    \begin{equation}
        \bm G 
        = 
        \sum_{r=0}^{q-1} \bm B_r^\top \bm B_r + \bm C^\top \bm C.
    \end{equation}
    Since \(\bm C^\top \bm C\) is positive semidefinite, it follows
    \begin{equation}
        \min \lambda(\bm G) \ge \sum_{r=0}^{q-1} \min \lambda (\bm B_r^\top \bm B_r).
    \end{equation}
    Each block admits the factorization
    \begin{equation}
        (\bm B_r)_{k,j} 
        = 
        2 \exp\left(-\frac{\beta}{2} t_j\right) T_j(x_{r,k}),
        \qquad
        x_{r,k} \coloneqq \cosh((r+kq)h_sh_t)
    \end{equation}
    for \(k, j = 0, \ldots, n_t\). Since 
    \(
        \exp(-\beta t_j/2) \ge \exp(-\beta\mu/2),
    \) 
    defining \((\bm V_r)_{k,j} \coloneqq T_j(x_{r,k})\) yields
    \begin{equation}
        \|\bm B_r z\|_2 
        \ge 
        2 
        \exp\left(-\frac{\beta\mu}{2}\right) 
        \|\bm V_r z\|_2,
        \qquad
        z \in \mathbb C^{n_t+1},
    \end{equation}
    and therefore
    \begin{equation}
        \min \lambda(\bm B_r^\top \bm B_r) 
        \ge 
        4 \exp(-\beta\mu) \frac{1}{\|\bm V_r^{-1}\|_2^2}.
    \end{equation}
    
    To bound \(\|\bm V_r^{-1}\|_2\), we use Proposition~\ref{prop:cheb_vander_inverse_estimate}
    \begin{equation}
        \|\bm V_r^{-1}\|_2 \le \sum_{k=0}^{n_t} \prod_{l=0, l \neq k}^{n_t} \frac{1 + x_{r,l}}{|x_{r,k} - x_{r,l}|}.
    \end{equation}
    Set
    \begin{equation}
        y_{r,k} \coloneqq (r+kq) h_sh_t,
    \end{equation}
    such that \(x_{r,k} = \cosh(y_{r,k})\). Using
    \begin{equation}
        1 + \cosh(y) = 2 \cosh^2(y/2)
    \end{equation}
    and
    \begin{equation}
        |\cosh(a)-\cosh(b)|
        =
        2
        \sinh\left(\frac{a+b}{2}\right)
        \sinh\left(\frac{|a-b|}{2}\right),
    \end{equation}
    we obtain
    \begin{equation}
        \frac{
            1+x_{r,l}
        }{
            |x_{r,k}-x_{r,l}|
        }
        =
        \frac{
            \cosh^2(y_{r,l}/2)
        }{
            \sinh\left(\frac{y_{r,k}+y_{r,l}}{2}\right)
            \sinh\left(\frac{|y_{r,k}-y_{r,l}|}{2}\right)
        }.
    \end{equation}
    Since \(r \ge 0\), we have
    \begin{equation}
        \frac{y_{r,k} + y_{r,l}}{2}
        \ge
        (k + l)\frac{q h_s h_t}{2}.
    \end{equation}
    Using
    \begin{equation}
        \sinh(mx) \ge m\sinh(x),
        \qquad
        m \in \mathbb N,
        \quad x \ge 0,
    \end{equation}
    we therefore get
    \begin{equation}
        \sinh\left(
            \frac{y_{r,k} + y_{r,l}}{2}
        \right)
        \ge
        (k+l)
        \sinh\left(\frac{qh_sh_t}{2}\right),
    \end{equation}
    and similarly
    \begin{equation}
        \sinh\left(
            \frac{|y_{r,k}-y_{r,l}|}{2}
        \right)
        \ge
        |k-l|
        \sinh\left(\frac{qh_sh_t}{2}\right).
    \end{equation}
    Combining both delivers
    \begin{equation}
        \frac{
            1+x_{r,l}
        }{
            |x_{r,k}-x_{r,l}|
        }
        \le
        \frac{
            \cosh^2(y_{r,l}/2)
        }{
            (k+l)|k-l|
            \sinh^2(q h_s h_t/2)
        }.
    \end{equation}
    Moreover, since \(i_{r,l}\le n_s\), we obtain
    \begin{equation}
        0
        \le
        y_{r,l}
        \le
        n_s h_s h_t
        =
        \frac{\beta\mu}{2n_t}.
    \end{equation}
    Using \(\cosh(x)\le \exp(x)\) gives
    \begin{equation}
        \prod_{l=0, l\neq k}^{n_t}
        \cosh^2(y_{r,l}/2)
        \le
        \exp\left(
            \sum_{l=0, l \neq k}^{n_t}
            y_{r,l}
        \right)
        \le
        \exp\left(\frac{\beta\mu}{2}\right).
    \end{equation}
    Consequently
    \begin{equation}
        \|\bm V_r^{-1}\|_2
        \le
        \frac{
            \exp(\beta\mu/2)
        }{
            \sinh^{2n_t}(q h_s h_t/2)
        }
        \sum_{k=0}^{n_t}
        \prod_{l=0, l\neq k}^{n_t}
        \frac{1}{
            (k+l)|k-l|
        }.
    \end{equation}
    
    Using Proposition~\ref{prop:binomial_type_formula} and the fact that \(\sinh(x) \ge x\) for \(x \ge 0\), we obtain
    \begin{equation}
        \|\bm V_r^{-1}\|_2
        \le
        \frac{
            \exp(\beta\mu/2)
        }{
            (2n_t)!
        }
        \left(
            \frac{
                8n_s n_t
            }{
                \beta\mu q
            }
        \right)^{2n_t}.
    \end{equation}
    It follows that
    \begin{equation}
        \min\lambda(\bm B_r^\top\bm B_r)
        \ge
        4\exp(-2\beta\mu)
        ((2n_t)!)^2
        \left(
            \frac{
                \beta\mu q
            }{
                8n_s n_t
            }
        \right)^{4n_t}.
    \end{equation}
    Summing over the \(q\) blocks gives
    \begin{equation}
        \label{eq:db_lower_bound_min_G}
        \min\lambda(\bm G)
        \ge
        4q\exp(-2\beta\mu)
        ((2n_t)!)^2
        \left(
            \frac{
                \beta\mu q
            }{
                8n_s n_t
            }
        \right)^{4n_t}.
    \end{equation}
    
    For the largest eigenvalue of \(\bm G\), we use
    \begin{equation}
        \max\lambda(\bm G)
        =
        \|\bm R\|_2^2
        \le
        \|\bm R\|_F^2.
    \end{equation}
    Since
    \begin{equation}
        R_{i,j}
        =
        \exp(-s_it_j)
        +
        \exp(-(\beta-s_i)t_j)
        \le
        2,
    \end{equation}
    we obtain
    \begin{equation}
        \max\lambda(\bm G)
        \le
        4(n_s+1)(n_t+1).
    \end{equation}
    
    Combining this estimate with \eqref{eq:db_lower_bound_min_G} yields
    \begin{equation}
        \kappa(\bm R)^2
        \le
        \frac{
            (n_s+1)(n_t+1)
        }{
            q
        }
        \frac{
            \exp(2\beta\mu)
        }{
            ((2n_t)!)^2
        }
        \left(
            \frac{
                8n_s n_t
            }{
                \beta\mu q
            }
        \right)^{4n_t}.
    \end{equation}
    Finally, using
    \begin{equation}
        (2n_t)!
        \ge
        \left(
            \frac{2n_t}{e}
        \right)^{2n_t},
    \end{equation}
    we obtain the asserted upper bound
    \begin{equation}
        \kappa(\bm R)
        \le
        \sqrt{
            \frac{
                (n_s+1)(n_t+1)
            }{
                q
            }
        }
        \exp(\beta\mu)
        \left(
            \frac{
                4e n_s
            }{
                \beta\mu q
            }
        \right)^{2n_t}.
    \end{equation}
\end{proof}

\section{Auxiliary results}

\begin{proposition}
    \label{prop:cheb_vander_inverse_estimate}
    Let pairwise distinct nodes \(x_0, \ldots, x_n \in (0,\infty)\) and consider the monomial and Chebyshev Vandermonde-type matrices
    \begin{equation}
        \bm V_{\mathrm{mon}} = \left(x_i^j\right)_{i,j=0,\ldots,n},
        \qquad
        \bm V_{\mathrm{cheb}} = \left(T_j(x_i)\right)_{i,j=0,\ldots,n},
    \end{equation}
    where \(T_j\) denotes the Chebyshev polynomial of the first kind of degree \(j\).
    Then, the \(2\)-norms of their inverses satisfy the following bound
    \begin{equation}
        \max \left\{\|\bm V_{\mathrm{mon}}^{-1}\|_2, \|\bm V_{\mathrm{cheb}}^{-1}\|_2\right\}
        \le 
        \sum_{i=0}^n \prod_{k=0, k \neq i}^n \frac{1 + x_k}{|x_i-x_k|}.
    \end{equation}
\end{proposition}
\begin{proof}
    Let \(\ell_i\) denote the \(i\)th Lagrange polynomial with respect to the nodes \(x_0,\ldots,x_n\)
    \begin{equation}
        \ell_i(x) =\prod_{k=0,k\neq i}^n \frac{x-x_k}{x_i-x_k}.
    \end{equation}
    We first consider its monomial representation, let \(c_{j,i}^{\mathrm{mon}} \in \mathbb R\) such that
    \begin{equation}
        \ell_i(x) = \sum_{j=0}^n c_{j,i}^{\mathrm{mon}} x^j.
    \end{equation}
    Here, the vector 
    \(
        (c_{0,i}^{\mathrm{mon}}, \ldots, c_{n,i}^{\mathrm{mon}})^\top
    \) 
    is the \(i\)th column vector of \(\bm V_{\mathrm{mon}}^{-1}\). 
    Let \(a_{j,i} \in \mathbb R\) denote the coefficients of the numerator, then
    \begin{equation}
        \prod_{k=0, k \neq i}^n (x - x_k) 
        = 
        \sum_{j=0}^n a_{j,i} x^j.
    \end{equation}
    Since all \(x_k > 0\), the coefficients \(a_{j,i}\) alternate in sign. Therefore, evaluating the polynomial at \(x = -1\) delivers
    \begin{equation}
        \sum_{j=0}^n |a_{j,i}|
        =
        \sum_{j=0}^n (-1)^{n-j} a_{j,i}
        =
        (-1)^n \sum_{j=0}^n a_{j,i} (-1)^{j}
        =
        (-1)^n \prod_{k=0, k \neq i}^n (-1 - x_k)
        =
        \prod_{k=0, k \neq i}^n (1 + x_k).
    \end{equation}
    Dividing by the denominator of \(\ell_i\) yields
    \begin{equation}
        \sum_{j=0}^n |c_{j,i}^{\mathrm{mon}}|
        =
        \prod_{k=0, k \neq i}^n \frac{1 + x_k}{|x_i - x_k|}.
    \end{equation}
    Next, we consider the Chebyshev representation, let \(c_{j,i}^{\mathrm{cheb}} \in \mathbb R\) such that
    \begin{equation}
        \label{eq:lag_cheby_repr}
        \ell_i(x) 
        = 
        \sum_{j=0}^n c_{j,i}^{\mathrm{cheb}} T_j(x).
    \end{equation}
    Now, the vector 
    \(
        (c_{0,i}^{\mathrm{cheb}},\ldots,c_{n,i})^\top
    \) 
    is the \(i\)th column vector of \(\bm V_\mathrm{cheb}^{-1}\). 
    Each monomial \(x^m\) admits an expansion in the Chebyshev basis
    \begin{equation}
        x^j 
        = 
        \sum_{k=0}^j b_{j,k} T_k(x),
        \qquad
        b_{j,k} \in \mathbb R.
    \end{equation}
    Setting \(x=\cos\theta\) and using \(T_j(\cos\theta)=\cos(j\theta)\), we obtain
    \begin{equation}
        \cos(\theta)^j 
        =
        \sum_{k=0}^j b_{j,k} \cos(k\theta).
    \end{equation}
    Indeed, \(b_{j,k} \ge 0\), and evaluating this identity at \(\theta=0\) yields
    \begin{equation}
        \sum_{k=0}^j b_{j,k} = 1.
    \end{equation}
    Combining both representations \(a_{j,i}\) and \(b_{j,k}\), we obtain
    \begin{equation}
    	\prod_{k=0, k \neq i}^n (x-x_k)
    	=
    	\sum_{j=0}^n a_{j,i} x^j
    	=
    	\sum_{j=0}^n a_{j,i} \sum_{k=0}^j b_{j,k} T_k(x)
        =
        \sum_{k=0}^n d_{k,i} T_k(x),
        \qquad
        d_{k,i} \coloneqq \sum_{j=k}^n a_{j,i} b_{j,k}.
    \end{equation}
    Now, using the triangle inequality for the Chebyshev coefficients \(d_j\) of \(\ell_i\), we get
    \begin{equation}
        \sum_{k=0}^n |d_{k,i}| 
        \le 
        \sum_{k=0}^n \sum_{j=k}^n |a_{j,i}| |b_{j,k}|
        =
        \sum_{j=0}^n |a_{j,i}| \sum_{k=0}^j |b_{j,k}|
        = 
        \sum_{j=0}^n |a_{j,i}|.
    \end{equation}
    This shows that passing from the monomial expansion to the Chebyshev expansion does not increase the absolute coefficient sum. Dividing again by the denominator of \(\ell_i\) yields
    \begin{equation}
        \sum_{j=0} |c_{j,i}^{\mathrm{cheb}}|
        \le 
        \frac{
            \prod_{k \neq i} (1+x_k)
        }
        {
            \prod_{k \neq i} |x_i-x_k|
        }.
    \end{equation}
    Finally, this delivers the asserted upper bound, since for either Vandermonde matrix \(\bm V\)
    \begin{equation}
        \|\bm V^{-1}\|_2
        \le
        \|\bm V^{-1}\|_F
        \le
        \sum_{i=0}^n \sum_{j=0}^n |(\bm V^{-1})_{j,i}|.
    \end{equation}
\end{proof}

\begin{proposition}
    \label{prop:binomial_type_formula}
    The following formula holds for \(n \in \mathbb N\)
    \begin{equation}
        \sum_{i=0}^n \prod_{k=0, k \neq i}^n \frac{1}{(i+k)|i-k|}
        =
        \frac{4^n}{(2n)!}.
    \end{equation}
\end{proposition}
\begin{proof}
    For \(i=0\), we have
    \begin{equation}
        \prod_{k=1} \frac{1}{k^2} = \frac{1}{(n!)^2}.
    \end{equation}
    For \(i \ge 1\), we separate the two products and rewrite
    \begin{equation}
        \prod_{k=0, k \neq i}^n \frac{1}{|i-k|} 
        =  
        \prod_{k=0}^{i-1} \frac{1}{(i-k)} \prod_{k=i+1}^{n} \frac{1}{(k-i)} = \frac{1}{i! (n-i)!},
    \end{equation}
    and
    \begin{equation}
        \prod_{k=0, k \neq i}^n \frac{1}{i+k} 
        = 
        2i \prod_{k=0}^n \frac{1}{i+k} 
        = 
        2i \frac{(i-1)!}{(n+i)!} 
        = 
        \frac{2 i!}{(n+i)!}.
    \end{equation}
    Together this yields
    \begin{equation}
        \prod_{k=0, k \neq i}^n \frac{1}{(i+k)|i-k|}
        =
        \frac{1}{i! (n-i)!} \frac{2 i!}{(n+i)!} = \frac{2}{(n-i)!(n+i)!}.
    \end{equation}
    Combining the case where \(i=0\) with the case where \(i \ge 1\), and summing this over \(i=0,\ldots, n\), we get
    \begin{equation}
        \begin{split}
            \sum_{i=0}^n \prod_{k=0, k \neq i}^n \frac{1}{(i+k)|i-k|}
            &= \frac{1}{(n!)^2} + 2 \sum_{i=1}^n \frac{1}{(n-i)!(n+i)!}\\
            &= \sum_{i=-n}^n \frac{1}{(n-i)!(n+i)!}
        \end{split}
    \end{equation}
    Setting \(j=n-i\), so that \(j=0,\ldots,2n\), yields
    \begin{equation}
        \begin{split}
            \sum_{i=-n}^n \frac{1}{(n-i)!(n+i)!} 
            &=
            \sum_{j=0}^{2n} \frac{1}{j! (2n-j)!} \\
            &=
            \frac{1}{(2n)!} \sum_{j=0}^{2n} {2n \choose j}\\
            &= \frac{4^n}{(2n)!}.
        \end{split}
    \end{equation}
\end{proof}

\newpage

\bibliographystyle{plainnat}
\bibliography{references}  

\end{document}

%% file: figures/general-conditioning.tex
\begin{figure}[h]
    \centering
    \includegraphics[width=0.95\textwidth]{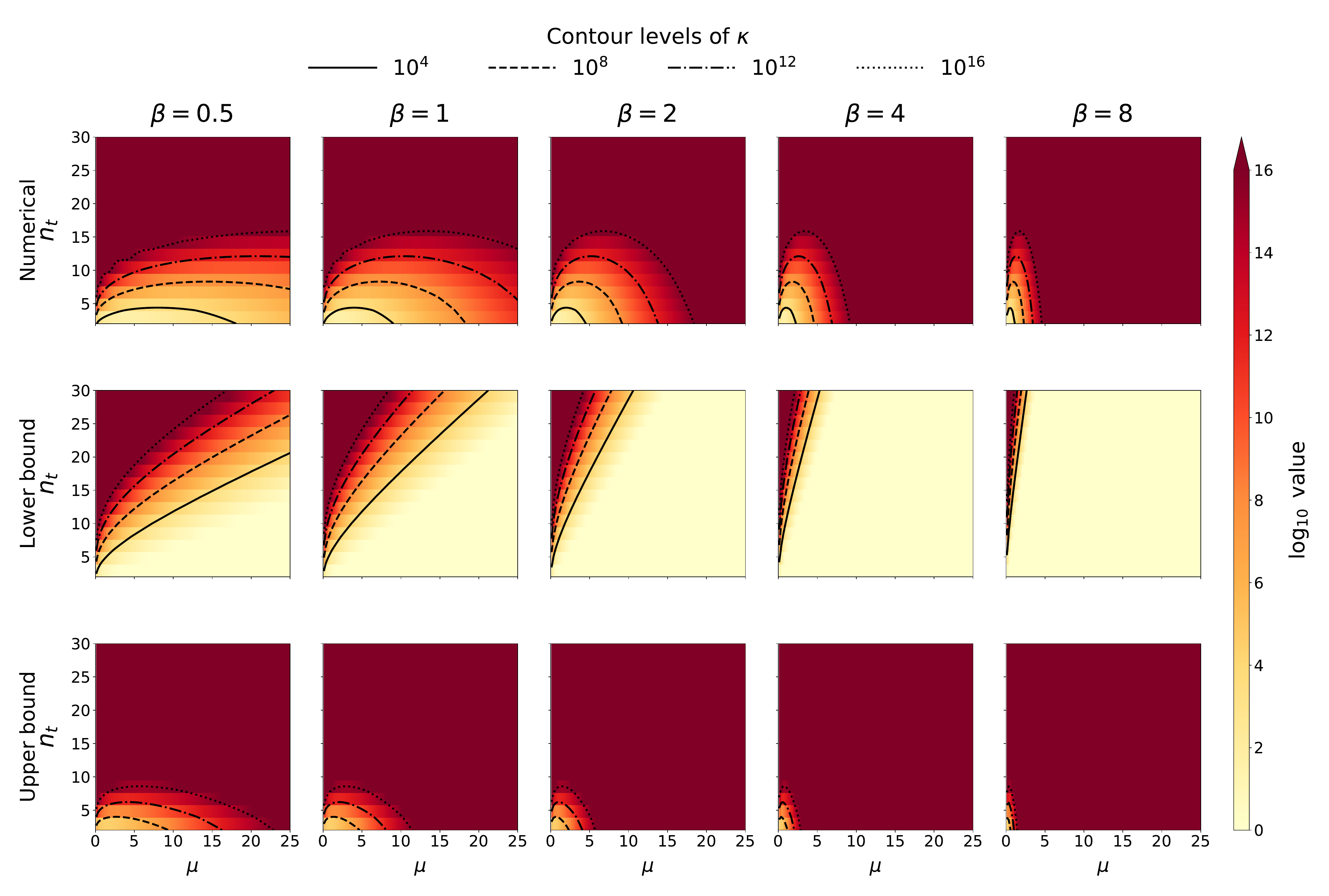}
    \caption{
        A demonstration that the numerical condition numbers satisfy the analytical bounds developed for the discretized general AC problem in the case where \(n_s=100\).
        From top to bottom, the rows show the numerical estimates of the condition number \(\kappa(\bm R)\) defined in \eqref{eq:general_AC_reg_mat} as well as the lower bound, and the upper bounds given in Theorem~\ref{theo:general_AC_bounds}.
        The columns correspond to different values of \(\beta\), while the contour lines indicate the levels \(10^4\), \(10^8\), \(10^{12}\), and \(10^{16}\), with \(10^{16}\) treated as numerical infinity. 
    }
    \label{fig:general_conditioning}
\end{figure}

%% file: figures/db-conditioning.tex
\begin{figure}[h]
    \centering
    \includegraphics[width=0.95\textwidth]{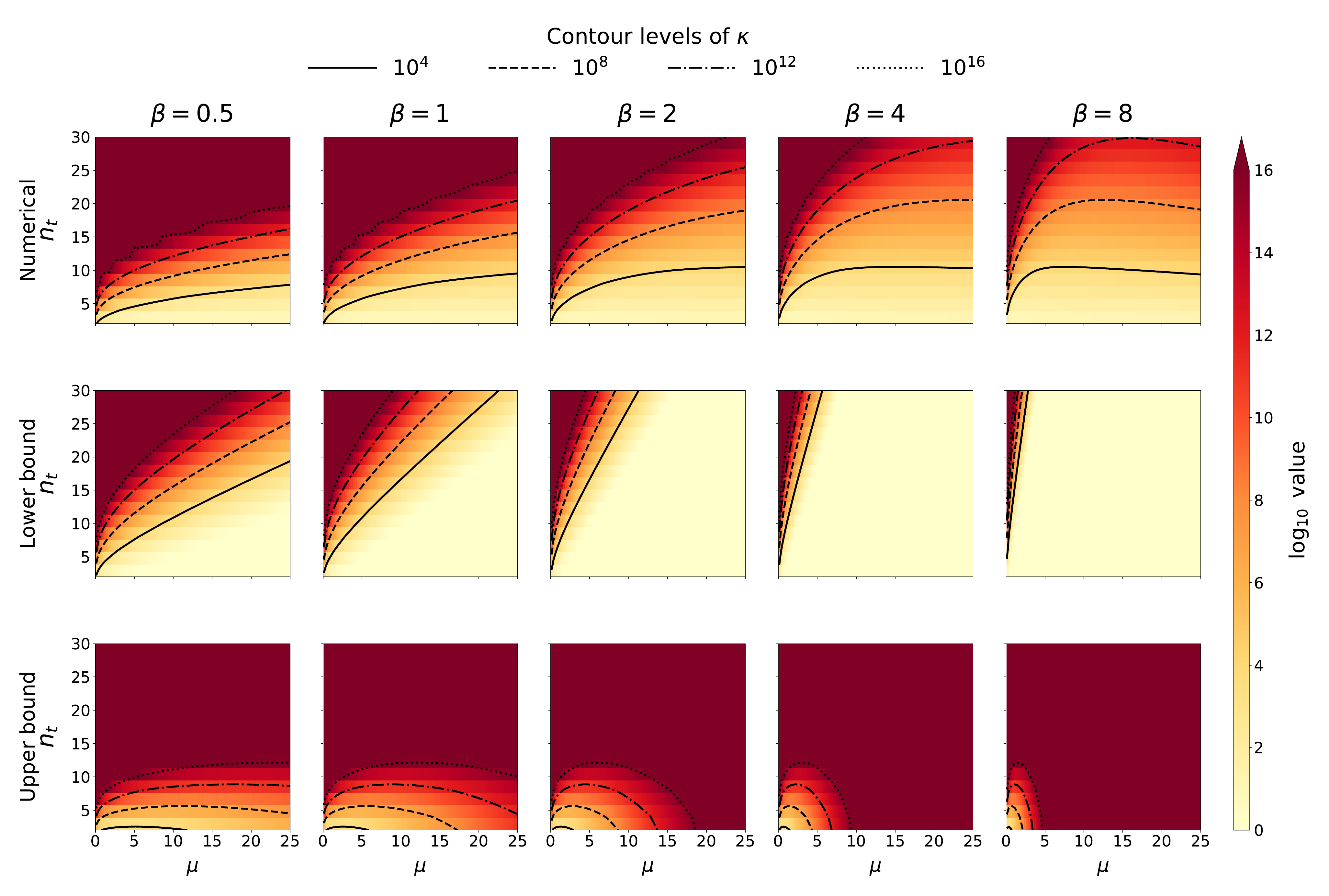}
    \caption{
        A demonstration that the numerical condition numbers satisfy the analytical bounds developed for the discretized detailed-balanced AC problem in the case where \(n_s = 100\).
        From top to bottom, the rows show the numerical estimates of the condition number \(\kappa(\bm R^{\text{db}})\) defined in \eqref{eq:db_AC_reg_mat} as well as the lower bound and the upper bounds given in Theorem~\ref{theo:db_AC_bounds}.
        The columns correspond to different values of \(\beta\), while the contour lines indicate the levels \(10^4\), \(10^8\), \(10^{12}\), and \(10^{16}\), with \(10^{16}\) treated as numerical infinity.
    }
    \label{fig:db_conditioning}
\end{figure}